\documentclass[10pt]{amsart}
\usepackage[margin=1.3in,marginparwidth=1.1in,marginparsep=0.1in]{geometry}
\usepackage{marginnote}
\usepackage{fancyhdr}
\usepackage{tikz-cd}
\usepackage{pgfplots}
\usepackage{pgfkeys}
\pgfplotsset{compat=1.12}
\usepgfplotslibrary{fillbetween}
\usetikzlibrary{intersections, pgfplots.fillbetween}
\usepackage[hidelinks]{hyperref}
\usepackage{setspace}
\theoremstyle{plain}
\newtheorem{thm}{Theorem}[section]
\newtheorem{claim}{Claim}
\newtheorem{lem}[thm]{Lemma}
\newtheorem{lemspecial}[thm]{Basic Inequality}
\newtheorem{prop}[thm]{Proposition}
\newtheorem{si}[thm]{Situation}
\newtheorem{conj}{Conjecture}
\newtheorem{cor}[thm]{Corollary}
\newtheorem{nt}[thm]{Notation}
\newtheorem{q}[thm]{Question}
\theoremstyle{definition}
\newtheorem{defn}[thm]{Definition}
\theoremstyle{remark}
\newtheorem{rem}[thm]{Remark}
\theoremstyle{definition}
\newtheorem{ex}[thm]{Example}

\newcommand{\DD}{\mathbb{D}}

\newcommand{\PP}{\mathbb{P}}
\newcommand{\QQ}{\mathbb{Q}}

\renewcommand{\SS}{\mathbb{S}}

\newcommand{\ZZ}{\mathbb{Z}}

\newcommand{\lag}{\langle}

\newcommand{\rag}{\rangle}
\newcommand{\qq}{\mathbb{Q}}

\newcommand{\la}{\lambda}
\newcommand{\aaa}{\alpha}

\newcommand{\ep}{\epsilon}

\newcommand{\ti}{\tilde}

\newcommand{\mO}{\mathscr{O}}
\newcommand{\mV}{\mathscr{V}}
\newcommand{\mU}{\mathscr{U}}
\newcommand{\mM}{\mathscr{M}}
\newcommand{\mQ}{\mathcal{Q}}
\newcommand{\cW}{\mathcal{W}}

\newcommand{\cM}{\mathcal{M}}

\newcommand{\sbs}{\subset}
\newcommand{\sg}{\sigma}
\newcommand{\mL}{\mathscr{L}}
\newcommand{\ov}{\overline}

\newcommand{\rr}{\rho_{g,k}}

\newcommand{\ot}{\otimes}

\newcommand{\cd}{\cdot}

\newcommand{\bfs}{\mathbf{s}}

\newcommand{\hh}{,\hdots,}

\usepackage{amsfonts}
\newcommand{\g}{\mathfrak{g}}
\usepackage{pdflscape}
\usepackage{amsmath,amssymb,latexsym,amsthm}
\usepackage{mathtools}
\usepackage{graphicx}
\usepackage{mathrsfs}
\usepackage{amscd}
\usepackage{color}

\newenvironment{st}[2][Step]{\begin{trivlist}
\item[\hskip \labelsep {\bfseries #1}\hskip \labelsep {\bfseries #2}]}{\end{trivlist}}

\DeclareMathOperator\rk{rank}
\DeclareMathOperator\Gr{Gr}

\DeclareMathOperator\fl{Fl}
\DeclareMathOperator\im{Im}

\DeclareMathOperator\ext{Ext}

\DeclareMathOperator\Pic{Pic}

\DeclareMathOperator\sy{Sym}

\DeclareMathOperator\sgn{sgn}

\DeclareMathOperator\id{Id}

\DeclareMathOperator\h{Hom}
\DeclareMathOperator\Sec{Sec}

\let \pr=\prime
\input xy
\xyoption{all}
\DeclareRobustCommand{\SkipTocEntry}[5]{}
\begin{document}
\title{The Strong Maximal Rank Conjecture and higher rank Brill--Noether theory}
\author{Ethan Cotterill, Adri\'an Alonso Gonzalo, and Naizhen Zhang}
\address[Ethan Cotterill]{Instituto de Matem\'atica, Universidade Federal Fluminense, Rua Prof Waldemar de Freitas, S/N, Campus do Gragoat\'a, CEP 24.210-201, Niter\'oi, RJ, Brazil}
\email{cotterill.ethan@gmail.com}
\address[Adri\'an Alonso Gonzalo]{Department of Mathematics, Universitat Aut\`onoma de Barcelona, 08193 Bellaterra (Barcelona), Spain}
\email{adrian.alonsogonzalo@gmail.com
}
\thanks{Naizhen Zhang is supported by the DFG Priority Programme 2026 ``Geometry at infinity". During much of the preparation of this work, he was supported by the Methusalem Project Pure 
Mathematics at KU Leuven.}
\address[Naizhen Zhang]{Institut f\"ur Differentialgeometrie, Leibniz Universit\"at Hannover, Welfengarten 1, 30167 Hannover, Germany}
\email{naizhen.zhang@math.uni-hannover.de}
\maketitle
\begin{abstract}
In this paper, we compute the cohomology class of certain ``special maximal-rank loci" originally defined by Aprodu and Farkas. By showing that such classes are nonzero, we are able to verify the non-emptiness portion of the Strong Maximal Rank Conjecture in a wide range of cases. As an application, we obtain new evidence for the existence portion of a well-known conjecture due to Bertram, Feinberg and independently Mukai in higher-rank Brill--Noether theory.    
\end{abstract}

\tableofcontents
\section{Introduction}
Ever since the inception of moduli of (complex) curves as an area of investigation in its own right, {\it linear series} have served as crucial tools for probing the {\it intrinsic} geometry of the moduli space via the {\it extrinsic} properties of a (variable) curve's embeddings in projective spaces. Classically, a linear series is defined by a vector subspace $V$ of holomorphic sections of a line bundle $L$; in that case, the celebrated {\it Brill--Noether} theorem of Griffiths and Harris gives a complete description of the space of series on a curve that is general in moduli. There are many interesting variations on the basic Brill--Noether paradigm; in this paper, we will be concerned with two of these. The first concerns the (dimension of) the vector space of hypersurfaces of fixed degree containing the image of $C$ under a linear series, while the second 
involves replacing $L$ by a vector bundle of some higher rank $e \geq 2$, and produces a theory of {\it rank-$e$ linear series}.

\medskip
For every integer $m \geq 2$, the dimension of the vector space of degree-$m$ hypersurfaces in $\mathbb{P}V^{\vee}$ containing the image of a curve $C$ under $(L,V)$ is determined by the rank of an $m$-fold {\it multiplication map} $v_m:\text{Sym}^m V \to H^0(C,L^{\ot m})$. Multiplication maps for {\it general} linear series on general curves are the focus of the {\it Maximal Rank Conjecture}, or MRC, now a theorem thanks to the work of Eric Larson. A strong form of the MRC, or SMRC, addresses the dimensionality of spaces of {\it special} linear series whose multiplication maps fail to be of maximal rank. The SMRC is already very much open when $m=2$ and much of our effort in this paper will be devoted to certifying the {\it positivity} of (classes of) quadratic SMRC loci of (traditional) linear series on a general curve. These are degeneracy loci in moduli spaces of linear series, in which the quadratic multiplication map $v_2$ fails to have maximal rank. While we manage to realize our classes as alternating sums of meromorphic functions in the {\it shifted Schur functions} of Okounkov and Olshanski, certifying the positivity of these expressions is subtle in general. 


\medskip
Our main result establishes the non-emptiness of quadratic SMRC loci (modulo a couple of explicit exceptions) whenever the dimension of the target of $v_2$ is at most 6 less than that of the domain, by verifying the positivity of the corresponding SMRC classes. By an SMRC class, we mean the image of the cohomology class of an SMRC locus under the Gysin map; see Section \ref{sec:review}.   

\begin{thm}[= Propositions~\ref{prop:asymp} and \ref{prop:positivity_for_small_N}]
Assume that the dimension of the domain of the multiplication map $v_2$ is $N-1$ more than that of the target, where $N \geq 1$. The SMRC class is strictly positive when $N \leq 2$
except when either
\begin{itemize}
    \item[(i)] $N=1$, and $(g,r,d) \in \{(1,2,3), (5,3,7)\}$; or
\item[(ii)] $N=2$, and $(g,r,d) \in \{(2,3,5), (7,4,10)\}$
\end{itemize}
and the SMRC class is unconditionally strictly positive whenever $3\leq N \leq 7$. Moreover, the SMRC class is always strictly positive whenever $r\ge 12N-2$. 
\end{thm}
Here positivity is measured by an intersection with a complementary power of the theta divisor.

\medskip
Quadratic SMRC loci are in fact intimately related to rank-two linear series; indeed, our original motivation for studying the SMRC is 
a celebrated conjecture of Bertram, Feinberg and Mukai (referred to hereafter as the {\it BFM conjecture}), which states the following: \begin{conj}\label{BFM}(Bertram-Feinberg-Mukai \cite{Ber,MCan})
\noindent Set $\rr:=3g-3-\binom{k+1}{2}$. On a general curve $C$ of genus $g\ge 2$, the moduli space of stable rank two vector bundles with canonical determinant and $k$ sections is non-empty, and has expected dimension $\rr$ whenever $\rr$ is non-negative. When $\rr<0$, the moduli space is empty.
\end{conj}
We shall refer to the first and second items as the  ``existence" and ``non-existence" portions of the conjecture, respectively. The non-existence portion is a theorem of Teixidor i Bigas \cite{M2}. Meanwhile, the existence portion of the conjecture has been verified for many cases: for small genera, see \cite{Ber} and \cite{MCGrass}; while for results asymptotic with respect to $(g,k)$, see \cite{M1}, \cite{LNP} and \cite{Zhang1}.  However, the general case remains very much open.

\medskip
In \cite{Zhang1}, the third author succeeded in verifying new instances of the BFM conjecture. In order to do so, he appealed to a theory of higher-rank {\it limit} linear series on reducible curves of compact type. In this paper, we instead study rank-two bundles with canonical determinant as extensions of line bundles by their Serre duals; the coboundary map in cohomology induced by the extension sequence is controlled by a quadratic multiplication map.

\medskip
With Theorem 1.1 in hand, we are able to extend the list of genera for which the existence portion of the BFM conjecture is known to hold:

\begin{thm}[= Corollary~\ref{cor:BFM_sharp}]
The existence portion of the BFM conjecture holds for $g \in \{14,17,18,19,22,26,31\}$.
\end{thm}
In fact, we prove more; namely, that in each of these instances, the stable bundle that we construct is realized by an extension of a line bundle of the minimal degree allowed by classical Brill--Noether theory. 
\addtocontents{toc}{\SkipTocEntry}
\subsection*{Roadmap}
The material following this introduction is structured as follows. In section~\ref{notations}, we list some notation related to linear series that we will use systematically throughout the entire paper. In section~\ref{sec:review}, we assemble all of the technical tools we will subsequently apply in calculating the classes of SMRC loci. We start by reviewing intersection theory on the Grassmann bundle $\mbox{Gr}(k,E)$. An important fact is that the Gysin (pushforward) morphism is induced by the {\it Lagrange--Sylvester symmetrizer} on the Chern roots of the bundle $E$; this is recalled in Lemma~\ref{gysin_map}. A key formula due to Laksov, Lascoux, and Thorup, which we recall in Theorem~\ref{thm:LLT1}, realizes the Chern polynomial of the symmetric square of a vector bundle $E$ as a linear combination of the Segre classes of $E$, whose coefficients are multiples of certain minors $d_I$ of an infinite matrix. The output of the LLT formula, in turn, turns out to be naturally related to the {\it shifted Schur functions} of Okounkov and Olshanski introduced in \cite{okounkov1997shifted}. For our SMRC class calculation, we specifically need to understand evaluations of shifted Schur functions evaluated along staircase partitions and their polynomiality properties; these are the focus of Proposition~\ref{prop:polynomiality}.

\medskip
Our exploration of the strong maximal rank conjecture starts in earnest in section~\ref{smrc}, where we explicitly describe quadratic SMRC loci $\mathcal{M}^r_d(C)$ as degeneracy loci for maps of vector bundles over the Picard variety of a general curve $C$. Subsection~\ref{injective_surjective} introduces a dichotomy between the {\it injective and surjective ranges}, depending upon how relatively large the (dimensions of the) source and target of quadratic multiplication are. Insofar as BFM loci are concerned, injective cases correspond to cases where the genus is small, while the surjective range describes the ``generic" case. 
Proposition~\ref{prop:triv_case} establishes that SMRC loci are always non-empty, and in fact contain excessively large components, in a particular regime of parameters $(g,r,d)$. We refer to these as {\it trivial instances} as they arise from the failure of the associated linear series to be very ample.

\medskip
In section~\ref{sec:computation}, we write down explicit formulae for the Gysin pushforwards to the Picard variety of the intersections of quadratic SMRC loci with complementary powers of the theta divisor. The basic shape of these formulae depends on whether the associated triple $(g,r,d)$ belongs to the injective or surjective range. 
In Lemma \ref{lem:m_small_value} of Subsection~\ref{subsec:F_g,r,d(m)_for_small_m}, we compute explicitly (the degrees of) SMRC classes using structure theorems for shifted symmetric functions in tandem with a computer program. 
Proposition~\ref{prop:asymp} establishes that for every fixed value of $N$, our SMRC class formulae are positive for all $r$ greater than an explicit cutoff function in $N$. The section culminates in Proposition~\ref{prop:positivity_for_small_N}. 

\medskip
Finally, section~\ref{sec:bfm} 
deals with the BFM conjecture and its relationship with the SMRC. 
In order to realize stable rank-two bundles with canonical determinant and prescribed numbers of sections as extensions $e$ of line bundles $L$ by their Serre duals, the crucial fact is that the map that sends an extension to its coboundary is dual to the quadratic multiplication map on sections of $L$. The other main point is that we have more control over the slope-stability of vector bundles arising from extensions in which $L$ is a minimal quotient of the resulting vector bundle. Often minimality of the quotient implies that $e$ is in the kernel of the dual of the quadratic multiplication map. If the quadratic multiplication map is surjective, any such $e$ is trivial; so stable rank-two bundles arise from (extensions of) line bundles that belong to quadratic SMRC loci.

\medskip
A nontrivial extension does not necessarily give rise to a stable bundle, however, so accordingly we develop additional tools for understanding when this happens. Proposition~\ref{prop:sec} gives a necessary geometric criterion: 
certain {\it secant divisors} to the image of $C$ under $|L|$ are obstructions to stability. On the other hand, Theorem~\ref{thm:nagata}, due to Nagata, gives (see Corollary~\ref{cor:min_quot}) 
an upper bound on the degree of the {\it minimal} quotient line bundle of a rank-two vector bundle with canonical determinant. It allows us to identify 
a critical range of possible degrees for line bundles. Proposition~\ref{prop:basic_case} establishes that a nontrivial extension of $L$ by its Serre dual gives rise to a stable bundle when its degree is at least $g$ and 
is {\it minimal} among degrees of all potentially feasible line bundles; its Corollary~\ref{cor:tri} ultimately leads to new cases of the (existence portion of the) BFM conjecture. In subsection~\ref{subsubsec: small_genera} we give a solution to the ``BFM existence problem" in the injective range via extensions, which gives an alternative to earlier work of Bertram and Feinberg. Afterwards, we establish a regime of parameters $(g,r,d)$ in which earlier results yield the existence portion of the BFM conjecture; this culminates in Corollary~\ref{cor:BFM_sharp}. We also discuss a significant case at the numerological border of the surjective range, that of $g=13, k=8$, in which our arguments are at present inconclusive.
Our Claim~\ref{BFM_g=13,k=8} establishes that the BFM conjecture holds in the $g=13, k=8$ case provided that on a general curve i) the multiplication map $\mu_2$ associated with a complete $g^5_{16}$ is always surjective; and ii) 
there exist {\it very ample} complete $g^6_{18}$ for which $\mu_2$ fails to be surjective.

\addtocontents{toc}{\SkipTocEntry}
\subsection*{Acknowledgements} We are grateful to Wouter Castryck, Renzo Cavalieri, Marc Coppens, Joe Harris, Thomas Lam, Alex Massarenti, Brian Osserman, and Richard Stanley for useful comments and conversations. Special thanks are due to Peter Newstead and Montserrat Teixidor i Bigas for the detailed corrections and suggestions they provided after reading an earlier version of this paper. Finally, we are grateful for the CNPq postdoctoral scheme that allowed the first and third authors to meet; and to the anonymous referee, who flagged several errors and whose suggestions have helped improve the organization and quality of exposition.

\section{Notation related to linear series}\label{notations}
\begin{nt}
Hereafter, $C$ will denote an irreducible smooth projective curve over an algebraically closed field $K$.
\end{nt}
\begin{nt}
$G^r_d(C)$ will denote the moduli scheme of linear series $\g^r_d$ over $C$.  
\end{nt}
\begin{nt}
$\mL$ will denote a Poincar\'e line bundle over $C\times\Pic^d(C)$. 
\end{nt}
Fix an effective, reduced divisor $D$ of $C$ of degree $\ge 2g-1-d$, and let  $D^{\pr}=D\times \Pic^d(C)$. The space $G^r_d(C)$ is naturally a closed subscheme of $\Gr(r+1,\pi_{2*}\mL(D^{\pr}))$. Indeed, it is the zero locus of the morphism 
\[\mU\hookrightarrow p^*\pi_{2*}\mL(D^{\pr})\to p^*\pi_{2*}(\mL(D^{\pr})|_{D^{\pr}})\] where $\mU$ is the tautological subbundle, $p:\Gr(r+1,\pi_{2*}\mL(D^{\pr}))\to \Pic^d(C)$ is the structure morphism and $\pi_2: C\times\Pic^d(C)\to \Pic^d(C)$ is the second projection. 
\begin{nt}
Let $i:G^r_d(C)\to\Gr(r+1,\pi_{2*}\mL(D^{\pr}))$ denote the corresponding closed immersion.
\end{nt}
\begin{nt}
Let $p^{\pr}:G^r_d(C)\to \Pic^d(C)$ 
denote the forgetful map $(L,V)\mapsto [L]$, and let $\pi^{\pr}_2:C\times G^r_d(C)\to G^r_d(C)$ denote the second projection morphism in this case.  
\end{nt}
\begin{nt}
Let $\mM$ denote the pull-back of $\mL$ to $C\times G^r_d(C)$ along $\id_C\times p^{\pr}$.
\end{nt}
\begin{nt}
Let $\mV$ denote the universal family over $G^r_d(C)$, i.e. the pull-back of $\mU$ to $G^r_d(C)$ along $i$.
\end{nt}
For the sake of convenience, we summarize the maps and spaces mentioned above in one commutative diagram:
\[\begin{tikzcd}
C\times G^r_d(C)\ar[r,"\id_C\times p^{\pr}"]\ar[d,"\pi^{\pr}_2"] &C\times\Pic^d(C)\ar[d,"\pi_2"]\\ G^r_d(C)\ar[rd,"i"]\ar[r,"p'"]&\Pic^d(C)\\&\Gr(r+1,\pi_{2*}\mL(D^{\pr}))\ar[u,"p"]
\end{tikzcd}\]  
\begin{nt}\label{symbols}
In numerical examples, $g$ will always denote the genus of the underlying curve, $r$ the (projective) dimension of a rank-one linear series, $d$ the degree of a line bundle, $\chi$ the Euler characteristic of a line bundle, $k$ the dimension of a rank 2 linear series, and
\[\rho=\rho(g,r,d):=g-(r+1)(r+g-d).\]
\end{nt}
\begin{nt}\label{nt:rho2}
Given non-negative integers $g,r,d$, we let 
\[D(g,r,d):=\rho(g,r,d)-1-\bigg|\binom{r+2}{2}-(1-g+2d)\bigg|.\]
When $\binom{r+2}{2}\ge 1-g+2d$, we also set 
\[N(g,r,d):=\binom{r+2}{2}-2d+g.\]
\end{nt}
\begin{nt}\label{ab-jac}
Fix a reference point $P_0$ on $C$. Let $w_{j}$ denote the corresponding class of $W_{g-j}$ for $j=1,\hdots,g$, given by the image of the map $u_j:\sy^jC\to J(C):Z\mapsto Z-jP_0$. It is a codimension-$(g-j)$ class. 
\end{nt}
In order to describe the Chern classes of the Poincar\'e line bundle, we single out certain cohomology classes of $C\times\Pic^d(C)$.
\begin{nt}\label{classes}
Let $\theta$ denote the class of the pull-back of the theta divisor, and let $\eta$ denote the pullback of the class of a point on $C$.
\end{nt}

\section{Intersection theory on Grassmann bundles, and (shifted) Schur functions}\label{sec:review}
In this section, we assemble all of the main ingredients required for our study of SMRC loci. We first review some well-known facts about intersection theory on Grassmann bundles.


\subsection{The Chow ring of a Grassmann bundle and the Gysin morphism}\label{grassmann_bundles}
Let $E$ be a locally-free sheaf of rank $n$ on a smooth projective variety $X$, and let $q:\Gr(k,E)\to X$ denote the natural morphism. The following structure theorem for the Chow ring of a Grassmann bundle is well-known.
\begin{thm}[\cite{Int}, 14.6.6]\label{Grass_bundle}
Let $\mU$ and $\mathcal{Q}$ denote the tautological subbundle and tautological quotient bundle of $q^{\ast}E$, respectively. The Chow ring $A(\Gr(k,E))$ is an algebra over $A(X)$ generated by the tautological classes
\[c_1(\mU),\hdots,c_k(\mU); \text{ and } c_1(\mathcal{Q}),\hdots,c_{n-k}(\mathcal{Q})\] modulo the relations $\sum_{i=0}^k c_i(\mU)\cd c_{k-i}(\mathcal{Q})=c_k(E)$. 
\end{thm}
Let $a_1,\hdots,a_k$ and $a_{k+1},\hdots,a_{n}$ denote the Chern roots of $\mU$ and $\mQ$, respectively. The Chern classes $c_i(\mU)$ and $c_j(\mQ)$ are elementary symmetric functions $e_i(a_1,\hdots,a_k)$ and $e_j(a_{k+1},\hdots,a_n)$. Consequently, if we think of $A(X)$ as a subring of $A(\Gr(k,E))$, $c_k(E)$ is a polynomial in $a_1,\hdots,a_{n}$ which is symmetric in $a_1,\hdots,a_k$ and in $a_{k+1},\hdots,a_{n}$ separately. The upshot is that we may express any intersection product involving $c_i(\mU)$, $c_j(\mQ)$ and $c_k(E)$ as a product of symmetric functions in the Chern roots of $\mU$ and $\mQ$. 
We will put this observation to work in writing down the Gysin map $q_{\ast}:A(\Gr(k,E))\to A(X)$. But first we recall another well-known fact, which we will also use.
\begin{thm}[\cite{gro_chern}, Theorem 3.1]\label{thm:flag}
Let $\fl(E)$ denote the complete flag bundle associated to a vector bundle $E$ of rank $n$ and $a_1,\hdots,a_n$ be the Chern roots of $E$. The Chow ring $A(\fl(E))$ is an $A(X)$-algebra, generated by the elements $a_1,\hdots,a_n$ modulo the relations
\[e_i(a_1,\hdots,a_n)=c_i(E), i=1,2,\hdots,n.\]
\end{thm}
\begin{lem}[\cite{pragacz1988enumerative}, Lemma 2.5]\label{gysin_map}
The Gysin morphism $q_*:A(\Gr(k,E))\to A(X)$ is induced by the map 
\[p:\ZZ[a_1,\hdots,a_n]^{\SS_k\times\SS_{n-k}}\to \ZZ[a_1,\hdots,a_n]^{\SS_n}:f(a_1,\hdots,a_n)\mapsto \sum_{\sigma\in\SS_n/\SS_{k}\times\SS_{n-k}}\sigma \bigg(\frac{f(a_1,\hdots,a_n)}{\prod_{\substack{1\le i\le k\\k+1\le j\le n}}(a_j-a_i)}\bigg)\]
where $\sigma$ acts on a polynomial by permuting the indices of the variables.
\end{lem}
\begin{rem}
The map $p$ is known as the \textit{Lagrange-Sylvester symmetrizer} (see also \cite{tu2017computing}). The attentive reader will notice that the the denominator on the right-hand side differs from that in the original formula of Pragacz by a power of $(-1)$. This disparity is explained by the fact that Pragacz stated the formula for the Gysin map with domain a Grassmannian of rank-$k$ quotient bundles, while our formula applies to the Gysin map whose domain is a Grassmannian of rank-$k$ sub-bundles. See also \cite[\S 4.1]{fulton2006schubert} for further discussion. 
\end{rem}

\subsection{Schur functions} 
Schur functions in $n$ variables are symmetric functions labeled by partitions of length at most $n$. Two equivalent conventions for Schur functions appear in the literature and are convenient for different purposes. We introduce both of them here and comment on their equivalence. 

\begin{defn}\label{defn:partition}
A \textit{partition} of length $n$ is a finite sequence $\mu=(\mu_{1},\hdots,\mu_{n})$ of non-negative integers arranged in non-increasing order. The \textit{conjugate} of a partition is a partition whose corresponding Young diagram is obtained from the original diagram by interchanging rows and columns.
\end{defn}

\begin{defn}\label{defn:Schur1}
Let $\la=(\la_1,\dots,\la_n)$ be a partition of length $n$. The Schur function $s_{\la}(x_1,\hdots,x_n)$ is the symmetric polynomial
\[\bfs_{\la}(x_1,\hdots,x_n)=\frac{\det(x_j^{\la_i+n-i})}{\prod_{1\le j<k\le n}(x_j-x_k)}.\]
\end{defn}
\begin{defn}\label{defn:Schur2}
Let $I=(i_1<\hdots<i_n)$ be a strictly increasing sequence of non-negative integers. Now set 
\[s_I(x_1,\hdots,x_n):=\det(s_{i_{\ell}-k+1})_{1\le k,\ell\le n}\]
where $s_j$ is the $j$-th coefficient in the formal expansion \[\prod_{i=1}^n\frac{1}{1-x_it}=\sum_{j=0}^{\infty}s_j(x_1,\hdots,x_n)t^j.\]
\end{defn}
\begin{defn}\label{defn:seq}
For $I=(i_1<\hdots<i_n)$, we call $n$ the \textit{length} of $I$ and denote $\ell(I)$. We also write $|I|=\sum_{k=1}^ni_k$. 
\end{defn}
\begin{nt}\label{nt:partition_seq}
For any given partition $\la$ of length $n$, and any given integer $m\ge n$, there is a unique strictly-increasing $m$-term sequence $(i_1\hh i_m)$ defined by
\[i_k=\la_{m-k+1}+k-1,\qquad\text{ for }k=1\hh m\]
in which we set $\la_k=0$ whenever $k>n$. We denote this sequence by $I_m(\la)$. Conversely, given a strictly-increasing sequence $I=(i_1\hh i_n)$ of non-negative integers, let
\[\la(I)=(i_n-n+1\hh i_{n-k+1}-n+k\hh i_1)\]
denote the corresponding partition of length at most $n$.
\end{nt}
Inasmuch as there is a bijection between strictly-increasing sequences of positive integers and partitions, definitions~\ref{defn:Schur1} and \ref{defn:Schur2} need to be reconciled. The Jacobi--Trudi lemma does the trick.

\begin{lem}[Lemma A.9.3 in \cite{Int}]\label{lem:JT}
Let $I=(i_1\hh i_n)$ be a strictly-increasing sequence of non-negative integers. We have 
\[\bfs_{\la(I)}(x_1,\hdots,x_n)=\det(s_{\la_k+\ell-k})_{1\le k,\ell\le n}=s_I(x_1,\hdots,x_n)\]
where $\la_1\hh\la_n$ are the parts of $\la(I)$.
\end{lem}
\begin{rem}\label{rem:Schur_val}
Note that $\deg(s_I(x_1,\hdots,x_n))=\sum_{k=1}^n i_k-\binom{n}{2}$, so that $s_{0,1,\hdots,n-1}(x_1,\hdots,x_n)=1$.
\end{rem}
\begin{defn}\label{defn:general}
More generally, given any finite non-necessarily-increasing sequence of non-negative integers  $I=(i_1,\hdots,i_n)$, we set
\[s_I(x_1,\hdots,x_n):=\frac{\det(x_{\ell}^{i_{n-k+1}})_{1\le k,\ell\le n}}{\prod_{1\le i<j\le n}(x_i-x_j)}.\] 
The fact that $s_I(x_1,\hdots,x_n)$ agrees with Definition~\ref{defn:Schur2} whenever $I$ is a strictly increasing sequence follows from Lemma~\ref{lem:JT}.
\end{defn}

\begin{nt}\label{nt:Schur_vb}
Given a rank-$n$ vector bundle $E$ with Chern roots $a_1,\hdots,a_n$, we will follow the convention in \cite[\S 14.5]{Int} and refer to $s_i(a_1,\hdots,a_n)$ as the $i$-th Segre class of $E^{\vee}$, also written as $s_i(-E)$. We set $s_I(E):=s_I(a_1,\hdots,a_n)$. Notice that the Jacobi-Trudi lemma expresses $s_I(E)$ as a determinant in terms of the classes $s_i(-E)$, as opposed to $s_i(E)$. 
\end{nt}
\begin{nt}\label{nt:Porteous}
Fix a positive integer $n$. Let $\SS=\{s_0,s_1,s_2,\hdots\}$ be a set of elements in some commutative ring and let $\mu$ denote a partition of length at most $n$. We set 
\[\Delta_{\mu}(\SS):=\det([s_{\mu_i+j-i}]_{1\le i,j\le n]}).\]
\end{nt}

\subsection{Combinatorial properties of the Lagrange-Sylvester symmetrizer}
We next review some well-known properties of the Lagrange-Sylvester symmetrizer that we will use. The first is a combinatorial formula that describes the action of $p$ on Schur functions. For a reference, see \cite{lascoux1988interpolation}.
\begin{lem}\label{lem:formula}
Given two sequences of non-negative integers $I=(i_1<\hdots<i_k)$ and $J=(j_{k+1}<\hdots<j_n)$, we have
\[p(s_I(a_1,\hdots,a_k)s_J(a_{k+1},\hdots,a_n))=(-1)^{k(n-k)}s_{J,I}(a_1,\hdots,a_n)\]
where $J,I$ denotes the concatenation of $J$ and $I$. 
\end{lem}
\begin{rem}
By definition, $p$ is clearly additive. Moreover, because
\[\ZZ[x_1,\hdots,x_n]^{\SS_k\times\SS_{n-k}}\cong\ZZ[x_1,\hdots,x_k]^{\SS_k}\ot_{\ZZ}\ZZ[x_{k+1},\hdots,x_n]^{\SS_{n-k}}\]
and the Schur polynomials form a $\ZZ$-basis for the ring of symmetric functions (see \cite[I.3.2]{macdonald_symm}), the formula in Lemma \ref{lem:formula} completely determines $p$.
\end{rem} 
\begin{cor}\label{cor:formula2}
The map $p$ satisfies the following properties:
\begin{enumerate}
\item $p(fg)=f\cd p(g)$ for every $\SS_n$-invariant polynomial $f$.
\item $p(s_I(a_1,\hdots,a_k))=(-1)^{k(n-k)}s_{0,\hdots,n-k-1,I}(a_1,\hdots,a_n)$ for every $k$-tuple $I$; in particular, $p(s_I(a_1,\hdots,a_k))=0$ whenever $i_1<n-k$.
\item $p(s_I(a_1,\hdots,a_k))$ is either a Schur polynomial in $a_1,\hdots,a_n$ or zero for every $k$-tuple $I$.
\end{enumerate}
\end{cor}
\begin{proof}
Whenever $f$ is symmetric, we have $\sg(fg)=f\sg(g)$ and thus $p(fg)=f\cd p(g)$, which is claim (1).
On the other hand, clearly $s_{0,\hdots,n-k-1}(a_{k+1},\hdots,a_n)=1$. The first part of claim (2) follows now from Lemma~\ref{lem:formula}. For the second part of claim (2), note that whenever $i_1<n-k$, we have $s_{0,\hdots,n-k-1,I}(a_1,\hdots,a_n)=\frac{F(a_1,\hdots,a_n)}{V_n(a_1,\hdots,a_n)}$, where $F$ is the determinant of a matrix with two identical rows and hence must be zero.  

\medskip
From (2), we know that either $p(s_I(a_1,\hdots,a_k))=0$ or else $0,1,\hdots,n-k-1,I$ is a strictly increasing sequence of non-negative integers. In the latter case, Lemma~\ref{lem:JT} establishes that $p(s_I(a_1,\hdots,a_k))$ is a Schur polynomial, which is claim (3). 
\end{proof}

For our main application, $X$ will be the Picard variety $\Pic^d(C)$ of a smooth curve $C$ of genus $g$ and $E$ will be the pushforward $\pi_{2 \ast}(\mL(D^{\pr}))$ of the twist of a Poincar\'e line bundle $\mL$ over $C\times\Pic^d(C)$, by the pullback $D^{\pr}$ of an effective divisor $D$ on $C$ of degree at least $\max\{2g-1-d,0\}$.
Since ultimately we are interested in whether certain cohomology classes over $\Pic^d(C)$ are non-zero, we work up to numerical equivalence. We will apply the following well-known result of Mattuck.
\begin{lem}[\cite{Mattuck3}, Example 14.4.5 of \cite{Int}]\label{Poincare}
Suppose $d>2g-2$. The Segre class of the pushforward of a Poincar\'e line bundle $\widetilde{\mL}$ is given by $s_k(\pi_{2 \ast}\widetilde{\mL})=w_k$ (see Notation~\ref{ab-jac}). Moreover, $k!w_k$ is numerically equivalent to $\theta^k$, where $\theta=w_1$ is the theta divisor class.
\end{lem}
In our context, the lemma says that the total Segre class, $s(\pi_{2 \ast}\mL(D^{\pr}))$, is equal to $e^{\theta}$ up to numerical equivalence and hence $c(\pi_{2 \ast}\mL(D^{\pr}))=s(\pi_{2 \ast}\mL(D^{\pr}))^{-1}=e^{-\theta}$, i.e., that $c_k(\pi_{2 \ast}\mL(D^{\pr}))=\frac{(-1)^k\theta^k}{k!}$. Thus, the relations in $A(\Gr(r+1,\pi_{2 \ast}\mL(D^{\pr})))$ are concisely expressed by the equations
\[
\sum_{i=0}^kc_i(\mU)c_{k-i}(\mathcal{Q})=\frac{(-1)^k\theta^k}{k!}.
\]
Moreover, it does no harm to assume that $\deg(D)=2g-1-d$, so that $E=\pi_{2 \ast}(\mL(D^{\pr}))$ is a rank-$g$ vector bundle over $\Pic^d(C)$. Consequently, we have $\theta=a_1+\hdots+a_g$ and $\theta^{g+1}=0$. Representing elements in $A(\Gr(k,E))$ by elements in $(\ZZ[a_1,\hdots,a_k]^{\SS_k}\ot\ZZ[a_{k+1},\hdots,a_{g}]^{\SS_{g-k}})[\theta]$, we may think of the Gysin morphism $q$ as a map
\[
\frac{(\ZZ[a_1,\hdots,a_k]^{\SS_k}\ot\ZZ[a_{k+1},\hdots,a_{g}]^{\SS_{g-k}})[\theta]}{\lag e_j(a_1,\hdots,a_g)-\frac{(-\theta)^j}{j!}\mid j\ge 1\rag+\lag\theta^{g+1}\rag}\to\ZZ[\theta]/\lag\theta^{g+1}\rag, \text{ via }[\sum_{j=0}^gf_j(a_1,\hdots,a_g)\theta^j]\mapsto [\sum_{j=0}^g p(f_j)\theta^j].
\]
Here $e_j(a_1,\hdots,a_g)$ is the $j$-th elementary symmetric function in $a_1,\hdots,a_g$; while $f_0,\hdots,f_g$ are arbitrary elements in $\ZZ[a_1,\hdots,a_k]^{\SS_k}\ot\ZZ[a_{k+1},\hdots,a_{g}]^{\SS_{g-k}}$; and $[\cdot]$ denotes an equivalence class.
\begin{lem}\label{lem:coeff1}
For any strictly increasing sequence $I=(i_1<\hdots<i_k)$ with $i_1\ge g-k$, we have
\[q([s_I(a_1,\hdots,a_k)])=[(-1)^{|I|-\binom{k}{2}}\frac{\prod_{1\le j<\ell\le k}(i_{\ell}-i_j)}{\prod_{j=1}^k(i_j-g+k)!}]\cd\theta^{|I|-\binom{k}{2}-k(g-k)}.\]
\end{lem}
\begin{proof}
Applying Corollary~\ref{cor:formula2}(2), we see that \[q([s_I(a_1,\hdots,a_k)])=[p(s_I(a_1,\hdots,a_k))]=[(-1)^{k(g-k)}s_{0,\hdots,g-k-1,I}(a_1,\hdots,a_g)].\] 
On the other hand, using the conventions established in Notation \ref{nt:Schur_vb} it is clear that
\[
s_{0,\hdots,g-k-1,I}(a_1,\hdots,a_g)\\
=\det\begin{bmatrix}
s_{i_1-g+k}(-E)&\hdots&s_{i_k-g+k}(-E)\\
\vdots &&\vdots\\
s_{i_1-g+1}(-E)&\hdots&s_{i_k-g+1}(-E)\\
\end{bmatrix}
\]
where $E=\pi_{2 \ast}\mL(D^{\pr})$. Substituting $s_k(-E)={(-\theta)^k\over k!}$, we get 
\[q([s_I(a_1,\hdots,a_k)])=\det\begin{bmatrix}
{1\over (i_1-g+k)!}&\hdots&{1\over (i_k-g+k)!}\\
\vdots &&\vdots\\
{1\over (i_1-g+1)!}&\hdots&{1\over (i_k-g+1)!}
\end{bmatrix}(-\theta)^{|I|-\binom{k}{2}-k(g-k)}.\]

\noindent It is well-known that the determinant 
\[\det\begin{bmatrix}
{1\over (i_1-g+k)!}&\hdots&{1\over (i_k-g+k)!}\\
\vdots &&\vdots\\
{1\over (i_1-g+1)!}&\hdots&{1\over (i_k-g+1)!}
\end{bmatrix}
\]
is a multiple of the Vandermonde determinant and can be computed as $\frac{\prod_{1\le j<\ell\le k}(i_{\ell}-i_j)}{\prod_{j=1}^k(i_j-g+k)!}$; cf. \cite[Ch.7, proof of Thm 4.4]{ACGH}.  The claim follows. 
\end{proof}

\subsection{The Littlewood--Richardson rule in terms of Schur functions}
Schur functions provide a convenient way of writing down the product of Schubert classes on a Grassmann bundle.
\begin{lem}\label{lem:lr}
Let $\Lambda_n$ be the ring of symmetric polynomials in $n$ variables. Fix $\la=(\underbrace{a,a,\hdots,a}_{n\text{ times}})$ with $a>0$ and suppose $\mu$ is any partition for which $|\mu|\le n$. We then have $s_{\la}\cdot s_{\mu}=s_{\la+\mu}$ in $\Lambda_n$.
\end{lem}
\begin{proof}
We begin by writing
\[s_{\la}\cdot s_{\mu}=\sum_{\nu\in\mathcal{P}_{|\la|+|\mu|,n}}c^{\nu}_{\la,\mu}s_{\nu}\]
where $\mathcal{P}_{|\la|+|\mu|,n}$ denotes the set of all partitions of length at most $n$ and size $|\la|+|\mu|$, and where the coefficients $c^{\nu}_{\la,\mu}$ are non-negative integers. Now suppose $c^{\nu}_{\la,\mu}\neq 0$. Since $\la$ is rectangular with $n$ rows, the length of $\nu$ is then necessarily exactly $n$ and $\nu-\la$ is still a partition.

\medskip
According to the Littlewood-Richardson rule, $c^{\nu}_{\la,\mu}$ counts the number of {\it Littlewood-Richardson tableaux}. These are semi-standard Young tableaux of shape $\nu-\la$ and weight $\mu$ whose characteristic property is that concatenating their reversed rows yields a word which is a lattice permutation; that is, in every initial part of the word, any number $i<n$ occurs at least as many times as $i+1$. The characteristic property immediately forces all the entries in the first row of such a tableau to be 1, and no further ones can occur in this tableau, since entries in every column form a strictly increasing sequence. Thus, $\nu_1-a=\mu_1$.

\medskip
Now remove the first row; the result is a Littlewood-Richardson tableau of shape $(\nu_2,\hdots,\nu_n)-(a,\hdots,a)$ and weight $(\mu_2,\hdots,\mu_n)$ on the alphabet $\{2,\hdots,n\}$. Applying the same argument as before, we get $\nu_2-a=\mu_2$. By induction, we conclude that $\nu=\mu$ and in particular $c^{\nu}_{\la,\mu}=1$.
\end{proof}
\begin{rem}
An alternative proof of the same result may be derived from Corollary 7.15.2 in \cite{stanley_fomin_1999}, which establishes that $c^{\nu}_{\la,\mu}$ is equal to the coefficient of $x^{\nu+\delta}$ in $V(x_1,\hdots,x_n)s_{\la}(x_1,\hdots,x_n)s_{\mu}(x_1,\hdots,x_n)$, where $\delta=(n-1,n-2,\hdots,1,0)$.

To wit, note that 
$s_{\la}(x_1,\hdots,x_n)=x_1^a\hdots x_n^a$ when $\la=(a,\hdots,a)$, as the only semi-standard Young tableau on the alphabet $\{1,2,\hdots,n\}$ is the one with every column equal to $(1,2,\hdots,n)^t$. It follows that
\[
\begin{aligned}
&V(x_1,\hdots,x_n)s_{\la}(x_1,\hdots,x_n)s_{\mu}(x_1,\hdots,x_n)=(x_1^ax_2^a\hdots x_n^a)\cd\det([x^{\mu_i+n-i}_j]_{1\le i,j\le n})\\
&=(x_1^ax_2^a\hdots x_n^a)\cd\sum_{\sg\in\SS_n}\sgn(\sg)\sg(x_1^{\mu_1+n-1}x_2^{\mu_2+n-2}\hdots x_n^{\mu_n})\\
&=\sum_{\sg\in\SS_n}\sgn(\sg)x_1^{a+\mu_{\sg(1)}+n-\sg(1)}x_2^{a+\mu_{\sg(2)}+n-\sg(2)}\hdots x_n^{a+\mu_{\sg(n)}}.
\end{aligned}
\]
Note also that $\nu+\delta$ is a strictly decreasing sequence. But $(a+\mu_{\sg(1)}+n-\sg(1),\hdots,a+\mu_{\sg(n)})$ is strictly decreasing if and only if the permutation $\sg \in\SS_n$ is the identity. So if $c^{\nu}_{\la,\mu}\neq 0$ then necessarily $\nu=\la+\mu$, in which case $c^{\nu}_{\la,\mu}=1$. 
\end{rem}
Using the alternative convention for Schur polynomials as in Definition \ref{defn:Schur2}, 
Lemma~\ref{lem:lr} becomes the following statement.
\begin{cor}\label{cor:Schur2}
Let $I=(i_1,\hdots,i_n)$ be an arbitrary strictly increasing sequence of non-negative integers and $J=(a,a+1,\hdots,a+n-1)$ ($a\ge 0$). We then have
\[s_I(x_1,\hdots,x_n)s_J(x_1,\hdots,x_n)=s_K(x_1,\hdots,x_n)\]
where $K=I+J-(0,1,\hdots,n-1)=I+(a,a,\hdots,a)$.
\end{cor}

\subsection{Characteristic classes of symmetric squares of vector bundles}
The symmetric square of the tautological subbundle on (the total space of) a Grassmann bundle plays a key role in the enumerative geometry of SMRC loci. Laksov, Lascoux and Thorup gave formulae (hereafter, the {\it LLT formulae}) that express the Chern and Segre classes of the symmetric square of a vector bundle in terms of its Segre classes. 
\begin{thm}[Proposition 2.8.4, 2.8.6, \cite{laksov1989giambelli}]\label{thm:LLT1}
Let $E$ be a vector bundle of rank $n$. We have
\begin{enumerate}
    \item $c(\sy^2E)=(-1)^{\binom{n}{2}}2^{-n(n-1)}\sum_I(-2)^{|I|}d_Is_I(E)$;
    \item $s((\sy^2E)^{\vee})=\sum_I\psi_Is_I(E)$.
\end{enumerate}
Here both sums are taken over all strictly increasing  sequences of non-negative integers. The coefficients $d_I$ are defined by
\[d_I=\det \begin{bmatrix}
(-1)^{i_1}\cd\binom{2n-1}{i_1}&(-1)^{i_1}\cd\binom{2n-3}{i_1}&\hdots&(-1)^{i_1}\cd\binom{1}{i_1}\\
(-1)^{i_2}\cd\binom{2n-1}{i_2}&(-1)^{i_2}\cd\binom{2n-3}{i_2}&\hdots&(-1)^{i_2}\cd\binom{1}{i_2}\\
\vdots&\vdots&&\vdots\\
(-1)^{i_n}\cd\binom{2n-1}{i_n}&(-1)^{i_n}\cd\binom{2n-3}{i_n}&\hdots&(-1)^{i_n}\cd\binom{1}{i_n}\\
\end{bmatrix}\]
and the coefficients $\psi_I$ are given by the following recursive relation:
\begin{enumerate}
 \item $\psi_i=2^i$, $\psi_{i,j}=\sum_{\ell=i+1}^{j}\binom{i+j}{\ell}$;
 \item $n\cd\psi_{i_1,\hdots,i_n}-2\sum_k\psi_{i_1,\hdots,i_k-1,\hdots,i_n}=0$ if $i_1>0$;
 \item $n\cd\psi_{0,i_2,\hdots,i_n}-2\sum_k\psi_{0,\hdots,i_k-1,\hdots,i_n}=\psi_{i_2,\hdots,i_n}$.\footnote{In (2) and (3), if $i_k-1=i_{k-1}$, $\psi_{i_1,\hdots,i_k-1,\hdots,i_n}$ is interpreted as 0.}
 \end{enumerate}
\end{thm}

\begin{rem}\label{B_matrix}
Let $\mathbf{B}$ denote the infinite matrix whose $(i,j)$-th entry is $\binom{i}{j}$, and let  
$\mathbf{B}^J_I$ denote the submatrix of $\mathbf{B}$ consisting of its $I$-th columns and $J$-th rows, where $\ell(I)=\ell(J)$.
We then have
\[d_I=
(-1)^{\deg(s_I)}\det(\mathbf{B}^{1,3\hh 2n-1}_{I}).\]
\end{rem}

As pointed out in \cite{laksov1989giambelli}, the convention that Laksov--Lascoux--Thorup use for the $i$-th Segre class, $s_i(E)$ differs from that of \cite{Int} by a factor of $(-1)^i$. In other words, the Segre classes in \cite{laksov1989giambelli} are the Segre classes of $E^{\vee}$ in \cite{Int}. Since we adopt the conventions of \cite{Int} in this paper, we rewrote the original formula of Laksov-Lascoux-Thorup accordingly.    

\subsection{Shifted Schur functions}
The LLT coefficients $d_I$ of the preceding subsection are closely related to the {\it shifted} Schur functions introduced by Okounkov and Olshanski in \cite{okounkov1997shifted}.

\begin{nt}[\cite{okounkov1997shifted}, (5.2), (5.3)]\label{nt:raising_fac}
Let $\mu$ be any Young diagram. Starting from the upper-left corner, identify the box in the $i$-th row, $j$-th column with the integer vector $(i,j)$. 
The {\it $\mu$-th generalized raising factorial of $n$} is defined by
\[
(n\upharpoonright\mu):=\prod_{(i,j)\in\mu}(n+j-i)=\prod_i\frac{(\mu_i+n-i)!}{(n-i)!}=\prod_i(\mu_i+n-i)_{\mu_i}.\footnote{Hereafter, $(a)_b$ will always denote the falling factorial $a(a-1)...(a-b+1)$.}
\]
\end{nt}
\begin{nt}[\cite{okounkov1997shifted},(11.1)]\label{nt:falling_fac}
Similarly, 
the {\it $\mu$-th generalized falling factorial of $n$} is defined by 
\[
(n\downharpoonright\mu):=\prod_{(i,j)\in\mu}(n-j+i)=\prod_i(n+i-1)_{\mu_i}.
\]
More generally, given any skew diagram $\mu/\nu$, we set 
$(n\downharpoonright\mu/\nu):=(n\downharpoonright\mu)\cd(n\downharpoonright\nu)^{-1}$.
\end{nt}
\begin{ex}
$(n\upharpoonright 1^n)=(n\downharpoonright (n))=n!$. In general, we have $(n\upharpoonright \mu)=(n\downharpoonright \mu^{\pr})$, where $\mu^{\pr}$ is the conjugate of $\mu$.  
\end{ex}

\begin{defn}
Let $\mu$ be a partition of length at most $n$. The {\it shifted Schur polynomial in $n$ variables with respect to $\mu$} is
\[s^*_{\mu}(x_1\hh x_n):=\frac{\det((x_i+n-i)_{i_j})_{1\le i,j\le n}}{\det((x_i+n-i)_j)_{1\le i,j\le n}},\]
where $(i_1,i_2\hh i_n)=I_n(\mu)$ (see Notation \ref{nt:partition_seq}).
\end{defn}
In \cite{okounkov1997shifted}, Okounkov and Olshanski introduced and studied the ring $\Lambda^*(n)$ of shifted polynomials in $n$ variables to be the algebra consisting of all $n$-variable polynomials that become symmetric after {\it shifting} variables according to
\[x^{\pr}_i=x_i-i+\text{const},\qquad i=1\hh n. \] 

The ring $\Lambda^*:=\projlim\Lambda^*(n)$ of {\it shifted symmetric functions}\footnote{Okounkov and 
Olshanski showed there is a natural map $\Lambda^*(n+1)\to \Lambda^*(n)$ defined by setting $x_{n+1}$ equal to 0, and that a well-defined limit over all $\Lambda^*(n)$ exists in the category of filtered algebras.} carries structures similar to those of the ring of classical symmetric functions.
\begin{thm}\label{thm:ok1}[Theorem 4.1, 4.2, \cite{okounkov1997shifted}]
There exists an involution $\omega:\Lambda^*\to \Lambda^*$ satisfying the following properties:
\begin{enumerate}
\item $\omega(f)(\lambda)=f(\lambda^{\pr})$, for all $f\in\Lambda^*$.
\item $\omega(s^*_{\mu})=s^*_{\mu^{\pr}}$.
\end{enumerate}
Here $\lambda^{\pr}$ denotes the conjugate of an arbitrary partition $\la$. 
\end{thm}
A key observation is that the coefficients $d_I$ can be written in terms of special values of shifted symmetric functions:
\begin{equation}\label{eqn:key}
d_I=(-1)^{|\la(I)|}\det(\mathbf{B}^{1,3\hh 2n-1}_{I})=\frac{(-1)^{|\la(I)|}}{\prod i_j!}\cd s^*_{\la(I)}(n\hh 1)\cd\det\Big([2i-1]\cd...\cd [2i-j+1]\Big)
\end{equation}
in which by convention we interpret the empty product as 1 and $\la(I)$ is as given in Notation \ref{nt:partition_seq}.  Notice also that 
\[\det\Big([2i-1]\cd...\cd [2i-j+1]\Big)=\det((2i-1)^{j-1})=V(1,3\hh 2n-1)=\prod_{1\le i<j\le n}2(j-i)=2^{\binom{n}{2}}\prod_{j=0}^{n-1}j!\]
where $V$ denotes the Vandermonde polynomial in $n$ variables, 
and recall from Remark \ref{rem:Schur_val} that 
\[|\lambda(I)|=\deg(s_I)=|I|-\binom{n}{2}.\]

As a result, we are primarily interested in special evaluations of shifted Schur functions along {\it staircase} partitions.
\begin{defn}\label{def:staircase_partition}
Given a nonnegative integer $r$, the $r$-th staircase partition is $\ep(r):=(r+1,\dots,1)$.
\end{defn}

\begin{thm}[Theorem 8.1, \cite{okounkov1997shifted}]\label{thm:schur_formula}
Let $\la\vdash K$, $\mu\vdash L$ be two partitions such that $K\le L$ and $\la\subset\mu$. We have
\[\frac{\dim \mu/\la}{\dim\mu}=\frac{s^*_{\la}(\mu)}{L(L-1)\cd\hdots\cd(L-K+1)}=\frac{s^*_{\la}(\mu)}{(L)_K}\]
where $\dim\mu/\la$ denotes the number of SYT of (skew) shape $\mu/\la$. 
\end{thm} 
\begin{cor}\label{cor:prob}
Given any non-negative integer $m\le\binom{r+2}{2}$, we have $\sum_{\la\vdash m}s^*_{\la}(\ep(r))=\binom{r+2}{2}_m$. 
\end{cor}
\begin{proof}
Note that $\binom{r+2}{2}$ is the size of $\ep(r)$. The corollary then follows from Theorem \ref{thm:schur_formula} together with the standard 
representation-theoretic fact that $\sum_{\la\vdash m}\frac{\dim \la\dim \ep(r)/\la}{\dim\ep(r)}=1$ (see, e.g., \cite[Example I.7.3]{macdonald_symm}):
\[\sum_{\la\vdash m}s^*_{\la}(\ep(r))=\sum_{\la\vdash m}\frac{\dim \la\dim \ep(r)/\la}{\dim\ep(r)}\cd\binom{r+2}{2}_m=\binom{r+2}{2}_m.\]
\end{proof}


In \cite{okounkov1997shifted}, Okounkov and Olshanski also compute a generating function for $s^*_{(k)}$, which they use to deduce a Jacobi--Trudi type formula for shifted Schur functions. Their result leads to the following proposition.

\begin{prop}\label{prop:polynomiality}
Given any partition $\la$, let $F_{\la}:\ZZ_{\ge -1}\to \QQ$ denote the function $r\mapsto s^*_{\la}(\ep(r))$. We have $F_{\la}(r)=g_{\la}(r)$ for some $g_{\la}(x)\in\QQ[x]$. Futhermore, the polynomials $g_{\la}(x)$ are such that
\begin{enumerate}
\item $g_{(k)}(x)=\frac{1}{2^kk!}(x+k+1)\cd(x+k)\cd...\cd (x-k+2)$;
\item $g_{\la}(-x-3)=g_{\la}(x)$; and
\item $\deg g_{\la}(x)=2|\la|$.
\end{enumerate}
\end{prop}

\begin{proof}
Our point of departure is the generating series for $s^*_{(k)}$ given in \cite[Thm 12.1]{okounkov1997shifted}:
\begin{equation}\label{s^*_k_generating_series}
H^{\ast}(u)=\sum_{k=0}^{\infty}\frac{s^{\ast}_{(k)}(x_1,x_2,...)}{(u)_k}=\prod_{i=1}^{\infty}\frac{u+i}{u+i-x_i}.
\end{equation} 
From \eqref{s^*_k_generating_series}, we deduce that
\begin{equation}\label{eqn:gen1}
\sum_{k=0}^{\infty}\frac{F_{(k)}(r)}{(u)_k}=\sum_{k=0}^{\infty}\frac{s^*_{(k)}(r+1,r\hh 1)}{(u)_k}=\prod_{i=1}^{r+1}\frac{u+i}{u+2i-2-r}
\end{equation}
for all $r\ge -1$. \footnote{To clarify, we have $s^{\ast}_{(k)}(r+1,r\hh 1)=s^{\ast}_{(k)}(0)$ when $r=-1$. Recall that an empty product is interpreted as 1.}

\medskip
{\flushleft Now} let $t=u^{-1}$. We then get ${1\over(u)_k}=t^k\prod_{i=1}^{k-1}(\sum_{n=0}^{\infty} (it)^n)$, and the left-hand side of (\ref{eqn:gen1}) can be written as 
\[\sum_{k=0}^{\infty}F_{(k)}(r)t^k\prod_{i=1}^{k-1}(\sum_{p=0}^{\infty} (it)^p)=F_{(0)}(r)+F_{(1)}(r)t+\sum_{k=2}^{\infty}(\sum_{p=2}^k(F_{(p)}(r)\cd(\sum_{\substack{(a_1\hh a_{p-1}):\\\sum a_i=k-p}}\prod_{j=1}^{p-1}j^{a_j})))t^k.\]
Notice that $\sum_{\substack{(a_1\hh a_{p-1}):\\\sum a_i=k-p}}\prod_{j=1}^{p-1}j^{a_j}=(1+2+...+(p-1))^{k-p}=\binom{p}{2}^{k-p}$. Accordingly, we have
\[
F_{(0)}(r)+F_{(1)}(r)t+\sum_{k=2}^{\infty}\Big(\sum_{p=2}^k \binom{p}{2}^{k-p}F_{(p)}(r)\Big)t^k=\prod_{i=1}^{r+1}\frac{u+i}{u+2i-2-r}.
\]

{\flushleft Meanwhile}, we have
\[\prod_{i=1}^{r+1}\frac{u+i}{u+2i-2-r}=\prod_{i=1}^{r+1} \frac{(1+it)}{(1-(r+2-2i)t)}.
\]
Let $G_r(t)$ denote the latter $t$-meromorphic function; 
note that
\[\log G_r(t)=\sum_{i=1}^{r+1} \log (1+it)-\sum_{i=1}^{r+1} \log (1+(2i-2-r)t).\]
It follows that ${d^k\over dt^k}\log G_r(t)|_{t=0}=(-1)^{k-1}(k-1)!(\sum_{i=1}^{r+1} i^k-\sum_{i=1}^{r+1} (2i-2-r)^k)$. 
Applying formulae of Faulhaber for sums of (alternating) consecutive powers as in \cite{knuth1993johann} and \cite{howard}, we deduce that
\[
{d^k\over dt^k}\log G_r(t)|_{t=0}= f_k(r):=\begin{cases}
\phi_k((r+2)(r+1))&\text{ if }k\text{ is odd}\\
(-1)^{r+k-1}\frac{(k-1)!}{2}E_k(r+2) &\text{ if }k\text{ is even}
\end{cases}
\]
where $\phi_k$ is a polynomial of degree ${k+1\over 2}$ such that $\phi_k(0)=0$, and $E_k(x)$ is the $k$-th Euler polynomial defined by $\sum_{k=0}^{\infty} E_k(z) \frac{x^k}{k!}= \frac{2e^{xz}}{e^x+1}$. 
At this stage, it is worth remarking that $f_k(-x-3)=f_k(x)$ always holds. Indeed, when $k$ is odd, the change of variable $x \mapsto -x-3$ preserves $(x+1)(x+2)$; while if $k$ is even, it is easy to check that the exponential generating series $\frac{e^{x(z+2)}}{e^x+1}+ \frac{e^{-x(z+2)}}{e^{-x}+1}$ and $\frac{e^{-x(z+1)}}{e^x+1}+ \frac{e^{x(z+1)}}{e^{-x}+1}$ for Euler polynomials $E_k(z+2)$ and $E_k(-z-1)$ (respectively) with even indices $k$ are equal.

\medskip
On the other hand, it follows from the Fa\`a di Bruno formula together with the fact that $G_r(0)=1$ that ${d^k\over dt^k} \log G_r(t)|_{t=0}$ can always be written as an integer linear combination of terms $G^{(\mu)}_r(0)$, where $\mu=(\mu_i)_i$ is a partition of $k$ and 
$G^{(\mu)}_r(0):=\prod_{i=1}^{\ell(\mu)} {d^{\mu_i}\over dt^{\mu_i}} G_r(t)|_{t=0}$. 
An induction on $k$ now shows that there exist polynomials $g_k(x)$ such that $G^{(k)}_r(0)=g_k(r)$ for all $r\ge -1$ and $\deg g_k(x)\le 2k$. And using the fact that
\[
F_{(0)}(r)+F_{(1)}(r)t+\sum_{p=2}^{\infty}\Big(\sum_{k=2}^p \binom{k}{2}^{p-k}F_{(k)}(r)\Big)t^p=\sum_{k=0}^{\infty}\frac{g_k(r)}{k!}t^k
\]
we conclude that 
the same polynomials $g_k(x)$ are such that $g_k(r)=F_{(k)}(r)$ for $r\ge -1$ and $\deg g_k(x)\le 2k$. Another induction on $k$ shows, moreover, that $g_k(-x-3)=g_k(x)$. 

\medskip
We now show that $\deg g_k(x)=2k$. To this end, note that when $0\le r+1<k$, \cite[Thm 3.1]{okounkov1997shifted} implies that $g_k(r)=0$. In particular, $x=-1\hh k-2$ are roots of $g_k(x)$. Using $g_k(-x-3)=g_k(x)$, it follows immediately that $g_k(x)=0$ for $x=-1-k,-k\hh -2$. On the other hand, since $\deg g_k(x)\le 2k$, we further get that $g_k(x)=c_k\cd(\prod_{i=-k-1}^{k-2}(x-i))$. To determine the coefficient $c_k$, we evaluate $g_k(x)$ at $x=k-1$ and apply Theorem~\ref{thm:schur_formula} to get 
\[c_k\cd (2k)!=s^*_{(k)}(k,k-1\hh 1)=\frac{f^{\ep(k-1)}}{f^{\ep(k)}}\cd\binom{k+1}{2}_k=(2k-1)!!.\]
Thus $c_k={(2k-1)!!\over(2k)!}={1\over 2^kk!}$, 
which proves item (1).

\medskip
To draw a similar conclusion for $F_{\la}(r)$ for an arbitrary partition $\la$, we apply the Jacobi--Trudi formula \cite[Thm 13.1]{okounkov1997shifted} for shifted Schur functions that relates arbitrary partitions to rectangular ones:
\begin{equation}\label{eqn:JT_shift}
F_{\mu}(r)=\det \bigg[\sum_{p=0}^{j-1}\binom{j-1}{p}(\mu_i-i+j-1)_pF_{\mu_i-i+j-p}(r) \bigg]_{1\le i,j\le\ell(\mu)}.
\end{equation}
Using the polynomiality of the $F_{(k)}$, we conclude immediately from \eqref{eqn:JT_shift} that there exist $g_{\la}(x)\in\QQ[x]$ for which $\deg g_{\la}(x)\le 2|\la|$, $g_{\la}(r)=F_{\la}(r)$ for $r\ge -1$ and $g_{\la}(-x-3)=g_{\la}(x)$, which proves item (2). 

\medskip
Finally, note that the $(i,j)$-th entry of the determinant in Equation~\ref{eqn:JT_shift} is a linear combination of $F_{(k)}(r)$'s, among which $F_{\mu_i-i+j}(r)$ is of highest degree in $r$. Thus, when we expand the determinant, we find that the 
leading coeffcient is precisely $\det(c_{\mu_i-i+j})_{1\le i,j\le\ell(\mu)}$, where $c_{s}={1\over 2^s s!}$. But 
\[\det(c_{\mu_i-i+j})_{1\le i,j\le\ell(\mu)}={1\over 2^{|\mu|}}\det({1\over(\mu_i-i+j)!})_{1\le i,j\le\ell(\mu)}={V(\mu_1,\mu_2-1\hh \mu_{\ell(\mu)}-\ell(\mu)+1)\over 2^{|\mu|}\cd\prod_{i=1}^{\ell(\mu)}(\mu_i+\ell(\mu)-i)!}>0\]
which proves item (3). 
\end{proof}

The final ingredients we shall need from \cite{okounkov1997shifted} are two branching rules for shifted Schur functions: 
\begin{lem}\label{OOthm9.1}
\[s^*_{\mu}(\ep(r))=\frac{\sum_{\mu^+}s^*_{\mu^+}(\ep(r))}{\binom{r+2}{2}-|\mu|}\]
where $\mu^+$ runs through all partitions of the form $\mu+(0\hh 0,\underbrace{1}_{j-\text{th}},0\hh 0)$, for some $j$.
\end{lem}
\begin{lem}\label{OOthm11.1}
\[s^*_{\mu}(\ep(r))=\sum_{\nu\prec\mu}s^*_{\nu}(\ep(r-1))(r+1 \downharpoonright \mu/\nu)\]
where $\nu \prec \mu$ if and only if $\mu_i \leq \nu_i \leq \mu_{i+1}$ for every $i$), and $(x \downharpoonright\mu/\nu)$ is the generalized falling factorial (cf. Notation \ref{nt:falling_fac}).
\end{lem}
\begin{rem}
Lemmas \ref{OOthm9.1} and \ref{OOthm11.1} are special cases of \cite[Thm 9.1]{okounkov1997shifted} and \cite[Thm 11.1]{okounkov1997shifted}, respectively. We state them in the above forms for future reference as well as for the convenience of the reader. 
\end{rem}

\section{The Strong Maximal Rank Conjecture for quadrics}\label{smrc}
\subsection{The statement of the conjecture}\label{subsec:statement}
The focus of this section is the Strong Maximal Rank Conjecture for quadrics, which we state as follows:
\begin{conj}[Strong Maximal Rank Conjecture]\label{conj:smrc}
Fix $g,r,d\ge 1$ such that $g-d+r\ge 0$ and $\rho(g,r,d)=g-(r+1)(g-d+r)\ge 0$. Let $C$ denote a (Brill--Noether and Petri) general curve of genus $g$.
\begin{enumerate}
    \item Suppose $D(g,r,d):=\rho-1-|\binom{r+2}{2}-(2d+1-g)|\ge 0$. The determinantal locus 
\[\mathcal{M}^r_d(C):=\{(L,V)\in G^r_d(C)|v_2:\sy^2 V\to H^0(L^{\ot 2})\text{ does not have maximal rank}\}\]
is non-empty and every irreducible component is at least $D(g,r,d)$-dimensional.
\item When $D(g,r,d)<0$, for all $\g^r_d$ in $G^r_d(C)$, the multiplication map $v_2$ has maximal rank.
\end{enumerate}
\end{conj}
Note that our formulation of the conjecture differs from the original version stated by Aprodu and Farkas in \cite{FSMRC}, where they impose the restriction $\rho<r-2$. Aprodu and Farkas conjectured that with this extra assumption $\mathcal{M}^r_d(C)$ should be exactly $D(g,r,d)$-dimensional. However, we are mainly concerned with the non-emptiness of $\mathcal{M}^r_d(C)$, and for our main application to rank two Brill--Noether theory it is necessary to remove the restriction $\rho<r-2$.\footnote{A relaxed version of the conjecture was also stated in \cite{FO}, removing the assumption that $\rho<r-2$.} However, we do not expect that in this generality the dimension of $\mathcal{M}^r_d(C)$ is exactly $D(g,r,d)$; see Example~\ref{ex:triv_ex}.

\medskip
Hereafter, we mainly focus on part (1) of the conjecture. First of all, we recall how to realize $\mathcal{M}^r_{d}(C)$ as the degeneracy locus of a vector bundle morphism $\phi:\mathbf{E}\to \mathbf{F}$ over $G^r_d(C)$. To this end, fix an effective divisor $D$ of degree $\max\{2g-1-d,0\}$ on $C$ and let $D^{\pr}$ and $D^{\pr\pr}$ denote its pullbacks to $C\times \Pic^d(C)$ and $C\times G^r_d(C)$, respectively. Recall also that $G^r_d(C)$ is the closed subscheme of $\Gr(r+1,\pi_{2*}(\mL(D^\pr)))$ that is the zero locus of the bundle map 
\[\mU\hookrightarrow p^{\ast}\pi_{2 \ast}\mL(D^{\pr})\to p^{\ast}\pi_{2 \ast}(\mL(D^{\pr})|_{D^{\pr}})\]
where $\mU$ is the tautological bundle of the relative Grassmannian. Let $\mV$ denote the pull-back of $\mU$ to $G^r_d(C)$. We collect the relevant morphisms in the following diagram:
\[
\begin{tikzcd}
C\times G^r_d(C)\ar[r,"\pi^{\pr}_{2}"]\ar[d,"\id_C\times p^{\pr} "]&G^r_d(C)\ar[d,"p^{\pr }"]\ar[r,"i"]&\Gr(r+1,\pi_{2 \ast}(\mL(D^{\pr})))\ar[dl,"p"]\\
C\times \Pic^d(C)\ar[r,"\pi_{2}"]&\Pic^d(C)&
\end{tikzcd}
\]

\noindent Set $\mathbf{E}:=\sy^2\mV$. As $\mV$ is a bundle of rank $r+1$, $\mathbf{E}$ is a bundle of rank $\binom{r+2}{2}$ over $G^r_d(C)$.

\medskip
Now let $\mL$ be a Poincare line bundle on $C\times\Pic^d(C)$ and let $\mM$ denote its pull-back to $C\times G^r_d(C)$. We claim that $\mathbf{F}:=\pi_{2\ast}(\mM^{\ot 2})$ is a rank $2d-g+1$ vector bundle over $G^r_d(C)$. Since $C$ is Petri-general and $r+1\ge 2$, by Petri's theorem we have  $h^1(\mM^{\ot 2}_s)=0$, for every $s\in G^r_d(C)$; see Lemma~\ref{lem:Petriapp} below. Thus $h^0(\mM^{\ot 2}_s)=2d-g+1$, and it follows from Grauert's theorem that $\pi_{2\ast}(\mM^{\ot 2})$ is locally free of rank $2d-g+1$ over $G^r_d(C)$ and $R^1\pi_{2\ast}\mM^{\ot 2}=0$. 
From the long exact sequence in cohomology, it follows that the sequence
\[0\to \pi_{2 \ast}(\mM^{\ot 2})\to \pi_{2\ast}(\mM^{\ot 2}(2D^{\pr\pr}))\to \pi_{2\ast}(\mM^{\ot 2}(2D^{\pr\pr})|_{2D^{\pr\pr}})\to 0\]
is exact.

\medskip
We now describe the morphism $\phi:\mathbf{E}\to \mathbf{F}$. To this end, let $\iota$ be the pull-back of $\mU\hookrightarrow p^{\ast}\pi_{2\ast}\mL(D^{\pr})$ to $G^r_d(C)$. Composing $\iota$ with the natural morphism 
\[p^{\pr \ast}\pi_{2 \ast}\mL(D^{\pr})\to \pi^{\pr}_{2\ast}(\id_C\times  p^{\pr})^{\ast}\mL(D^{\pr})=\pi^{\pr}_{2\ast}\mM(D^{\pr\pr})\] we get a morphism $\mV\to \pi^{\pr}_{2\ast}\mM(D^{\pr\pr})$, and hence $\mV^{\ot 2}\to ( \pi^{\pr}_{2 \ast}\mM(D^{\pr\pr}))^{\ot2}$.  Composing the latter with the natural morphism $(\pi^{\pr}_{2 \ast}\mM(D^{\pr}))^{\ot2}\to\pi^{\pr}_{2*}(\mM^{\ot2}(2D^{\pr}))$ yields $\mV^{\ot 2}\to \pi^{\pr}_{2 \ast}(\mM^{\ot2}(2D^{\pr}))$. 

\medskip
From the definition of $G^r_d(C)$, $\mV\to p^{\pr \ast} \pi_{2 \ast}(\mL(D^{\pr}))\to p^{\pr \ast}\pi_{2 \ast}(\mL(D^{\pr})|_{D^{\pr}})$ is the zero map and thus $\mV^{\ot 2}\to \pi'_{2 \ast}(\mM^{\ot2}(2D^{\pr}))\to\pi^{\pr}_{2\ast}(\mM^{\ot 2}(2D^{\pr\pr})|_{2D^{\pr\pr}}) $ is zero. Hence, the morphism $\mV^{\ot2}\to\pi^{\pr}_{2\ast}(\mM^{\ot2}(2D^{\pr\pr}))$ factors through $\pi^{\pr}_{2 \ast}(\mM^{\ot2})=\mathbf{F}$. The morphism $\phi:\mathbf{E}\to\mathbf{F}$ is then induced via descent by the morphism of locally-free sheaves $\mV^{\ot2}\to\mathbf{F}$. As a consequence, $\mathcal{M}^r_d(C)$ is scheme-theoretically defined as the degeneracy locus of $\phi$  (cf. \cite{FKos}, \cite{FO2}), which is closed in $G^r_d(C)$. 

\medskip
The following result is well-known; but for completeness and for lack of an adequate reference, we include its proof.

\begin{lem}\label{lem:Petriapp}
Let $L$ be a line bundle on a Petri-general curve $C$ for which $h^0(L)\ge 2$; then $h^1(L^2)=0$.
\end{lem} 
\begin{proof}
By the Gieseker-Petri theorem, the multiplication map $\mu:H^0(L)\ot H^0(\omega\ot L^{-1})\to H^0(\omega)$ is injective. By Serre duality, we have $H^1(L^{\ot 2})\cong H^0(\omega\ot L^{-2})^{\vee}\cong \h(L,\omega\ot L^{-1})^{\vee}$. If $h^1(L^{\ot 2})\neq 0$, then $\h(L,\omega_C\ot L^{-1})^{\vee}$ and its dual $\h(L,\omega_C\ot L^{-1})$ are nonzero. Thus, there exists a nonzero and hence injective morphism from $L$ to $\omega\ot L^{-1}$, and the induced linear map $H^0(L)\to H^0(\omega\ot L^{-1})$ is multiplication by some rational function $f\in K(C)$. Since $h^0(L)\ge 2$, there exist linearly independent sections $s,s^{\pr}$ of $L$ for which $s\ot (f\cd s^{\pr})- s'\ot (f\cd s)$ is a non-zero element in the kernel of $\mu$, a contradiction. 
\end{proof}

\begin{defn}
We call $\mathcal{M}^r_d(C)$ the \textit{SMRC locus for quadrics}.
\end{defn}
\addtocontents{toc}{\SkipTocEntry}
\subsection*{The regime $d\ge g$}
Hereafter, we focus on the special case where $d\ge g$. As we shall see later, this will be the case relevant to the Bertram--Feinberg--Mukai conjecture; it is also cohomologically simpler than the general case. 
\begin{lem}
When $d\ge g$, we have $\pi^{\pr}_{2 \ast}(\mM^{\ot 2})=p^{\pr \ast}\pi_{2 \ast}(\mL^{\ot 2})$.
\end{lem}
\begin{proof}
This essentially follows from standard properties of cohomology and base change, as applied to the fibered square
\[\begin{tikzcd}
C\times G^r_d(C)\ar[r,"\id_C\times p^{\pr}"]\ar[d,"\pi^{\pr}_2"] &C\times\Pic^d(C)\ar[d,"\pi_2"]\\ G^r_d(C)\ar[r,"p^{\pr}"]&\Pic^d(C)
\end{tikzcd}\] 
and the invertible sheaf $\mL^{\ot 2}$ on $C\times\Pic^d(C)$. Recall first $\pi^{\pr}_{2*}(\mM^{\ot 2})=\pi^{\pr}_{2*}(\id_C\times p^{\pr})^*\mL^{\ot2}$. Note that $d\ge g$ implies $h^1(C,\mL^{\ot 2}_s)=0$ for all $s\in\Pic^d(C)$. Consequently, cohomology and base change commute for $\mL^{\ot2}$ in degree 0, i.e. $\pi^{\pr}_{2*}(\id_C\times p^{\pr})^*\mL^{\ot2}=p^{\pr \ast}\pi_{2*}(\mL^{\ot 2})$.  
\end{proof}
\begin{rem}
The same argument 
shows that $\pi^{\pr}_{2 \ast}(\mM(D^{\pr\pr}))=p^{\pr \ast}\pi_{2 \ast}(\mL(D^{\pr}))$ without making any assumption on $d$ relative to $g$. However, without the hypothesis $d\ge  g$, we do not have $\mathbf{F}=p^{\pr \ast}\pi_{2 \ast}(\mL^{\ot 2})$ and in general $\pi_{2 \ast}(\mL^{\ot 2})$ is not locally free. 
\end{rem}
Hereafter we adopt the following approach to studying the non-emptiness of the SMRC locus.

\medskip
{\flushleft \bf Strategy.} Let $\ti{\mathbb{D}}$ be the class of the degeneracy locus of $\ti{\phi}:\sy^2\mU\to p^{\ast}\pi_{2\ast}(\mL^{\ot 2}(2D^{\pr}))$ in $A(\Gr(r+1,\pi_{2\ast}\mL(D^{\pr})))$ (in particular $\phi$ is the pullback of $\ti{\phi}$ along $i$), and let $\mathbb{D}$ be the class of the degeneracy locus of $\phi$ in $A(G^r_d(C))$. If $[G^r_d(C)]\cdot \ti{\mathbb{D}}\neq 0$, then $\mathbb{D}\neq 0$; in other words, $\mathcal{M}^r_d(C)$ is non-empty. Furthermore, if $q([G^r_d(C)]\cdot \ti{\mathbb{D}})\neq 0$, then $\mathcal{M}^r_d(C)$ is non-empty.\\

We briefly outline how to calculate $\ti{\mathbb{D}}$ and $[G^r_d(C)]$. The former may be calculated using Porteous' formula, for which we need to determine the Chern classes of $\sy^2\mU$ and $p^{\ast}\pi_{2\ast}(\mL^{\ot 2}(2D^{\pr}))$. To handle $\sy^2\mU$, we apply the LLT formulae of Theorem~\ref{thm:LLT1}. 
To calculate the Chern classes of $\pi_{2 \ast}(\mL^{\ot2}(2D^{\pr}))$, on the other hand, we apply the Grothendieck--Riemann--Roch formula. Following Theorem~\ref{Grass_bundle}, (the pullbacks of) these classes naturally belong to 
$A(\Gr(r+1,\pi_{2 \ast}\mL(D^{\pr})))$. 

\medskip
Meanwhile, $G^r_d(C)$ is the zero locus of the bundle map 
\[\mU\hookrightarrow p^{\ast}\pi_{2 \ast}\mL(D^{\pr})\to p^{\ast}\pi_{2 \ast}(\mL(D^{\pr})|_{D^{\pr}})\]
so $[G^r_d(C)]$ itself may be computed using Porteous' formula. 
Recall that the Chern classes of $\mU$ are among the generators of $A(\Gr(r+1,\pi_{2*}\mL(D^{\pr})))$ over $A(\Pic^d(C))$. On the other hand, $\pi_{2*}(\mL(D^{\pr})|_{D^{\pr}})\cong\pi_{2*}(\mL(D^{\pr})/\mL)$ is a direct sum of line bundles algebraically equivalent to the trivial bundle \cite[Sec. VII.2]{ACGH}. So, modulo algebraic equivalence (and hence up to numerical equivalence), $\pi_{2 \ast}(\mL(D^{\pr})|_{D^{\pr}})$ has trivial Chern classes. It follows that 
\[c(p^{\ast}\pi_{2\ast}(\mL(D^{\pr})|_{D^{\pr}})-\mU)=s(\mU).\] 

Assume $\deg(D)=2g-1-d\ge 0$. Porteous' formula yields 
\begin{equation}\label{grdclass}
[G^r_d(C)]=\det([s_{2g-1-d+j-i}(\mU)]_{1\le i,j\le r+1})=(-1)^{(r+1)(2g-1-d)}s_{2g-1-d,\hdots,2g-1-d+r}(\mU).
\end{equation}
\begin{ex}\label{ex:Wrd}
Applying Lemma \ref{lem:coeff1} in tandem with \eqref{grdclass}, we recover the well-known expression for the class of $W^r_d(C)$ inside $\Pic^d(C)$:
\[[W^r_d(C)]=q_*([G^r_d(C)])=[\frac{\prod_{1\le j<\ell\le r+1}(i_{\ell}-i_j)}{\prod_{j=1}^{r+1}(i_j-g+r+1)!}]\cd\theta^{|I|-\binom{r+1}{2}-(r+1)(g-r-1)}\]
where $I=(i_1,i_2,\hdots,i_{r+1})=(2g-1-d,2g-d,\hdots,2g-1-d+r)$. Equivalently,
\[[W^r_d(C)]=\frac{\prod_{\aaa=0}^r\aaa!}{\prod_{\aaa=0}^r(r+g-d+\aaa)!}\theta^{(r+1)(r+g-d)}.\]
\end{ex}
\addtocontents{toc}{\SkipTocEntry}
\subsection*{The injective and surjective ranges}\label{injective_surjective}
Hereafter, we shall refer to cases for which $\binom{r+2}{2}< 2d-g+1$ as \textbf{\textit{cases within the injective range}} and cases for which $\binom{r+2}{2} \ge 2d-g+1$ as \textbf{\textit{cases within the surjective range}}. Since the surjective range will be the main focus of the remainder of the paper, we record the most salient inequalities operative in that range as follows:
\begin{lemspecial}\label{lem:num_assumption}
Let $N=\binom{r+2}{2}-(2d-g)\ge 1$ and suppose $D(g,r,d) \ge 0$. We have
\begin{equation}\label{eqn:num}
    1\le N+(r+1)(r+g-d)\le g\le d\le\min\{r+g,2d-g\}\le 2d-g=\binom{r+2}{2}-N.
\end{equation}
\end{lemspecial}

\subsection{Some known cases}\label{known_cases}
Various cases of part (2) of Conjecture \ref{conj:smrc} have already been established.
\begin{enumerate}
    \item Aprodu and Farkas show in \cite[Prop. 5.7]{FSMRC} that when $\rho<r-2$ and $r+g-d=0$, part (2) of Conjecture~\ref{conj:smrc} holds; that is, the special maximal-rank locus $\mathcal{M}^r_d(C)$ is empty. In other words, the multiplication map $v_2:\sy^2H^0(L)\to H^0(L^{\ot 2})$ is surjective for every degree $d$ line bundle $L$. (Note that in this case, $d\ge 2g+3$, and consequently every $\g^r_d$ is a complete linear series.) In fact, the same is true for the $n$-th multiplication map $v_n:\sy^nH^0(L)\to H^0(L^{\ot n})$, for all $n\ge 2$. 
    \item Farkas and Ortega show in \cite[Prop. 2.3]{FO2} that when $d\le g+1$ and $r=3$, part (2) of Conjecture~\ref{conj:smrc} holds; that is, $\mathcal{M}^3_d(C)$ is empty. In other words, $v_2|_V:\sy^2 V\to H^0(L^{\ot 2})$ is injective, for any $4$-dimensional subspace $V$ of sections of a degree-$d$ line bundle $L$. On the other hand, Teixidor i Bigas shows in \cite{T03} that whenever $d\le g+1$, part (2) of Conjecture~\ref{conj:smrc} holds for {\it complete} linear series.
    \item More recently, two separate groups \cite{JP,LOTZ18} have shown (working independently, and using different methods) that for $r=6$, $g=22,23$ and $d=g+3$, the map $v_2$ is injective for every line bundle $L$ of degree $d$ on a general curve of genus $g$. This means, in particular, that the respective loci where $v_2$ fails to be injective determine {\it divisors} in the space of linear series $\mathcal{G}^6_d$ and in $\mathcal{M}_g$. This potentially has important implications for the birational geometry of the moduli space of curves in genus 22 and 23. Indeed, in \cite{FKos,FHurw}), Farkas computed the classes of the corresponding virtual divisors, and showed that their (virtual) {\it slopes} are strictly less than $6+ \frac{12}{g+1}$. To conclude, it remains to establish that the natural forgetful projections from $\mathcal{G}^6_d$ to $\mathcal{M}_{22}$ and $\mathcal{M}_{23}$ are {\it generically finite} along SMRC divisors.
\end{enumerate}

\subsection{Excess components of the SMRC locus}\label{excess_components}
We now turn to part (1) of Conjecture \ref{conj:smrc}. We describe a family of cases within the surjective range for which the associated SMRC loci are always non-empty. 
In fact, it is easy to see that whenever they exist, {\it non-very ample} linear series contribute components of larger-than-expected dimension to SMRC loci. 
\begin{prop}\label{prop:triv_case}
Let $g,r,d$ be non-negative integers satisfying the following conditions: 
\begin{enumerate}
\item $r\ge r+g-d\ge 0$;
    \item $\rho(g,r+1,d)<0\le\rho(g,r,d-1)$; and
    \item $\binom{r+1}{2}\ge 1+2d-g$. 
\end{enumerate}
For every Brill--Noether general curve $C$, the SMRC locus $\mathcal{M}^r_d(C)$ has an excessively large component.
\end{prop}

\begin{proof}
The second condition, coupled with the fact that $C$ is Brill--Noether general, implies that there exist $\g^r_{d-1}$'s (and hence $\g^r_d$'s) on $C$, and that every $\g^r_d$ (and every $\g^r_{d-1}$) is a complete linear series. Now let $|L|$ denote a $\g^r_{d-1}$ on $C$. Then for every point $P$,
$|L(P)|$ is a $\g^r_d$ on $C$ which is not base-point free, since the inclusion $H^0(L)\to H^0(L(P))$ is an isomorphism. 

\medskip
On the other hand, condition (1) implies that $\deg(L^2(2P))\ge 2g$, and hence $|L^2(2P)|$ must be base-point free. But the image of the multiplication map $\nu_2:H^0(L(P))^{\ot2}\to H^0(L^2(2P))$ is contained in 
$H^0(L^{\ot 2})$. The upshot is that $\nu_2$ is not surjective, and $|L(P)|$ belongs to $\mathcal{M}^r_d(C)$. 

\medskip
Finally, condition (3) implies that $(g,r,d)$ falls in the surjective range, and moreover, that
\[\rho(g,r,d-1)>\rho(g,r,d)-1-\bigg(\binom{r+2}{2}-(1+2d-g)\bigg).\]
But by construction, $\mathcal{M}^r_d(C)$ contains an isomorphic image of $G^r_{d-1}(C)$, of dimension $\rho(g,r,d-1)$. It follows, in particular, that $\mathcal{M}^r_d(C)$ contains an excessively large component.
\end{proof}

\begin{rem}
Whenever $(g,r,d)$ satisfies the condition in Proposition~\ref{prop:triv_case}, we say that $(g,r,d)$ is a \textit{trivial instance} within the surjective range. This naturally raises the question of whether non-empty non-trivial SMRC loci exist.
\end{rem}

\begin{ex}\label{ex:triv_ex}
It is easy to check that $(g,r,d)=(16,7,22)$ is a trivial instance in the surjective range. In this case $\rho(16,7,21)=0$, so there are finitely many $\g^7_{21}$'s on $C$. 
Each one of these generates a 1-dimensional family of linear series in $G^7_{22}(C)$ along which $\nu_2$ fails to be surjective. 
\end{ex}

More generally, whenever $G^r_d(C)$ contains a large component consisting of non-very ample linear series, $\mathcal{M}^r_d(C)$ may have excessively large components.    
\begin{lem}\label{lem:non_very_ample}
Suppose $|L|$ is a non-very ample $\g^r_d$ on a general genus $g$ curve such that $d\ge g+1$. Then $\nu_2$ is not surjective.
\end{lem}
\begin{proof}
In light of Proposition~\ref{prop:triv_case}, it suffices to consider the case where $|L|$ is base-point free. As $|L|$ is not very ample, there is some pair of points $R_1,R_2\in C$ (not necessarily distinct) for which 
\[h^0(L(-R_1-R_2))=h^0(L(-R_1))=h^0(L(-R_2))=r.\]
Consequently, $\im(\nu_2)$ has no section which vanishes to order exactly 1 at $R_1$ and does not vanish at $R_2$. But $d\ge g+1$, so $|L^{\ot 2}|$ is very ample. It follows that $\nu_2$ is not surjective. 
\end{proof}

\begin{ex}\label{ex:non_very_ample}
Suppose $(g,r,d)=(13,6,18)$. Every $\g^6_{18}$ on a general genus 13 curve is a base-point free complete linear series $|L|$ with $h^1(L)=1$. It follows that every $\g^6_{18}$ is of the form $|\omega_C(-Z)|$, where $Z$ is an effective divisor of degree 6. Since a general curve of genus 13 has no $\g^1_6$, we have $\omega_C(-Z_1)\cong \omega_C(-Z_2)$ if and only if $Z_1=Z_2$.

\medskip
On the other hand, a $\g^6_{18}$ 
fails to be very ample if and only if it contains a $\g^5_{16}$. In other words, $|\omega_C(-Z)|$ fails to be very ample if and only there exist $R_1,R_2\in C$ such that $|\mO(Z+R_1+R_2)|$ is a $\g^1_8$. Now let $I\subset G^1_8(C)\times \sy^6C$ denote the incidence variety of degree 6 effective divisors contained in a $\g^1_8$. It is easy to see that the projection $I \rightarrow G^1_8(C)$ over the curve $G^1_8(C)$ has one-dimensional fibers, so $I$ is irreducible and 2-dimensional. Furthermore, a 6-tuple of points can be contained in at most one $g^1_8$, because there is no $\g^2_{10}$ on a general curve of genus $13$. It follows that $\mathcal{M}^r_d(C)$ has a component that is at least 2-dimensional. However, $\rho(13,6,18)-1-|\binom{8}{2}-(1+36-13)|=1$. So $\mathcal{M}^r_d(C)$ has an excessively large component in this case.
\end{ex}

In light of Proposition~\ref{prop:triv_case} and Lemma~\ref{lem:non_very_ample}, the following question is fundamental.

\begin{q}\label{basic_question}
When $D(g,r,d)\ge 0$, does $\mathcal{M}^r_d(C)$ contain a very ample $\g^r_d$?
\end{q}
It is worth mentioning that in the original SMRC proposed by Aprodu and Farkas, the condition $\rho<r-2$ implies that every $\g^r_d$ on a general curve is very ample. (Indeed, this follows from \cite[Thm 0.1]{FSec}.) So answers to question~\ref{basic_question} will naturally extend the work of Aprodu and Farkas. As we will see later, affirmative answers to question~\ref{basic_question} will also lead to solutions to existence problems in higher-rank Brill-Noether theory.

\section{Enumerative calculations for SMRC loci}\label{sec:computation}
\subsection{Chern classes of $\sy^2\mU$ and $\pi_{2*}(\mL^{\ot2})$, and the degeneracy class of $\ti{\phi}$}\label{subsec:chern_sq}
\begin{si}
In this section, we make the running assumption that $d\ge g$. 
\end{si}
Our first goal is to determine the Chern classes of $\pi_{2 \ast}(\mL^{\ot2}(2D^{\pr}))$. Recall that $D^{\pr}$ is the pull-back of an effective divisor $D$ on $C$ to $C\times\Pic^d(C)$ and $\mL(D^{\pr})$ is a Poincar\'e line bundle on $C\times\Pic^{d+\deg(D)}(C)$. 
So it suffices to compute the Chern classes of $\pi_{2 \ast}(\mL^{\ot2})$.  Notice also under our running assumption that $h^1(C,\mL^{\ot2}_s)=0$ for all $s\in\Pic^d(C)$, and hence $R^i\pi_{2 \ast}\mL^{\ot 2}=0$ for every $i>0$. Grothendieck--Riemann--Roch now yields 
\begin{equation}\label{GRR}
\mbox{ch}(\pi_{2*}\mL^{\ot2})\cd \mbox{td}(\Pic^d(C))=\pi_{2*}(\mbox{ch}(\mL^{\ot2})\cd \mbox{td}(C\times\Pic^d(C))).
\end{equation}
The Todd class of an abelian variety is trivial, so $\mbox{td}(C\times\Pic^d(C))$ is the pull-back of $\mbox{td}(C)$. Accordingly, \eqref{GRR} reduces to
\[\mbox{ch}(\pi_{2*}\mL^{\ot2}))=\pi_{2*}(\mbox{ch}(\mL^{\ot2})\cd \mbox{td}(C))=\pi_{2*}\bigg(\mbox{ch}(\mL^{\ot2})\cd\bigg(1+{1\over 2}c_1(\mathcal{T}_C)\bigg)\bigg).\] 

We still need to compute $\mbox{ch}(\mL^{\ot 2})$, or equivalently $c_1(\mL^{\ot2})=2c_1(\mL)$. The latter is, however, given explicitly in \cite[Ch. VIII]{ACGH}. The upshot is that up to numerical equivalence
\[ch(\pi_{2*}(\mL^{\ot2}))=\pi_{2*}\bigg(ch(\mL)^2\cd\bigg(1+{1\over 2}c_1(\mathcal{T}_C)\bigg)\bigg)=(1-g+2d)-4\theta.\]
Equivalently, we have $c(\pi_{2*}\mL^{\ot2})=e^{-4\theta}$.

\medskip
Applying the LLT formulae of Theorem~\ref{thm:LLT1} in our context, we obtain 
\[
\begin{split}
s_k(\sy^2\mU)&=(-1)^k\sum_I\psi_Is_I(\mU) \text{ and}\\
c_k(\sy^2\mU)&=(-1)^{\binom{r+1}{2}}2^{-r(r+1)}\sum_I(-2)^{|I|}d_Is_I(\mU)
\end{split}
\]
where the summations are over all degree-$k$ Schur functions in the Chern roots of $\mU$, and the $\psi_I$ and $d_I$ are particular combinatorial coefficients defined by determinantal formulae; see 
Theorem \ref{thm:LLT1} for their precise definition.

\medskip
We now return to the class calculation initiated in section~\ref{subsec:statement}. By Porteous' formula, the class $\ti{\DD}$ in $A(Gr(r+1, \pi_{2 \ast}\mL(D^{\pr})))$ of the 
locus over which the vector bundle map $\ti{\phi}:\sy^2\mU\to p^{\ast}\pi_{2 \ast}(\mL^{\ot 2}(2D^{\pr}))$ fails to be of maximal rank is given by  
\[\Delta^{(\binom{r+2}{2}-c)}_{1+2d-g-c}(\SS)\]
in which $c=\min\{\binom{r+2}{2},1+2d-g\}-1$, $\SS$ denotes the set of variables 
\[\SS=\{c_n(p^{\ast}\pi_{2\ast}(\mL^{\ot 2}(2D^{\pr}))-\sy^2\mU)\mid n\ge 0\}=\bigg\{(-1)^n\sum_{I:|I|\le\binom{r+1}{2}+n}\frac{(4\theta)^{n-\deg(s_I)}\psi_I}{(n-\deg(s_I))!}s_I(\mU)\mid n\ge 0 \bigg\}\]
and $\Delta^{(p)}_{q}(\cdot)=\Delta_{(\underbrace{q,\hdots,q}_{p\text{ times}})}(\cdot)$ (see Notation~\ref{nt:Porteous}).


\subsection{Classes of SMRC loci} Combining our formulae from sections~\ref{subsec:statement} and \ref{subsec:chern_sq}, we deduce that up to numerical equivalence the class of $\mathcal{M}^r_d(C)$ is given by
\begin{equation}\label{main_eq}
S:=[G^r_d(C)]\cdot \ti{\mathbb{D}}=(-1)^{(r+1)(2g-1-d)}s_{2g-1-d,\hdots,2g-1-d+r}(\mU)\cdot\Delta_{(1+2d-g-c)}^{\binom{r+2}{2}-c}(\SS) 
\end{equation} 
where $\SS=\{(-1)^n\sum_{I:|I|\le\binom{r+1}{2}+n}\frac{(4\theta)^{n-\deg(s_I)}\psi_I}{(n-\deg(s_I))!}s_I(\mU)\mid n\ge 0\}$ and $c=\min\{\binom{r+2}{2},1+2d-g\}-1$. 
As we mentioned in section~\ref{smrc}, there is now a basic dichotomy depending on the sign of $\binom{r+2}{2}-(1+2d-g)$. 

\subsubsection{The surjective range: $\binom{r+2}{2}\ge 1+2d-g$.} 
This is the case of primary interest to us. In this case, $1+2d-g-c=1$ and by applying \cite[Lemma 14.5.1]{Int} we may rewrite \eqref{main_eq} as
\[S:=(-1)^{(r+1)(2g-1-d)+N}s_{2g-1-d,\hdots,2g-1-d+r}(\mU)\cdot\Delta^{(1)}_{N}(\SS^{-1})
\]
where $N=\binom{r+2}{2}-c=\binom{r+2}{2}-2d+g$ and $\SS^{-1}$ denotes the set of variables
\[\{s_n(p^*\pi_{2*}(\mL^{\ot 2}(2D^{\pr}))-\sy^2\mU)\mid n\ge 0\}=\bigg\{\sum_{k=0}^nc_{n-k}(\sy^2\mU)\frac{(4\theta)^k}{k!}\bigg| n\ge 0\bigg\}.\]
Simplifying, we find that the class of the SMRC locus is given by
\begin{equation}\label{eqn:int_prod}
S=(-1)^{(r+1)(d+1)+N}s_{J}(\mU)\cdot\sum_{I:|I|\le\binom{r+1}{2}+N}\frac{(-1)^{\deg(s_I)}2^{2N-|I|}d_I}{(N-\deg(s_I))!}s_I(\mU)\cd\theta^{N-\deg(s_I)}
\end{equation}
where $J=(2g-1-d,2g-d,\hdots,2g-1-d+r)$. To go further, will explicitly rewrite the products $s_J(\mU) \cd s_I(\mU)$ using Lemma~\ref{lem:lr}.
Now assume $2g-1-d\ge 0$. Let $a=2g-1-d$ and $n=r+1$. Applying Lemma~\ref{lem:coeff1} and Corollary~\ref{cor:Schur2}  to the intersection product $S$ given by equation~(\ref{eqn:int_prod}) as well as one of the LLT formulae, we get\footnote{Recall (Remark \ref{rem:Schur_val}) that $\deg(s_I)=|I|-\binom{r+1}{2}$. In particular, $N-\deg(s_I)\ge 0$ for all $I:|I|\le \binom{r+1}{2}+N$.} 
\begin{equation}\label{eqn:surj_range}
q_*(S)=(-1)^N\Big[\sum_{I:|I|\le\binom{r+1}{2}+N}\frac{2^{2N-|I|}d_I}{(N-\deg(s_I))!}\cd\frac{\prod_{1\le j<\ell\le r+1}(i_{\ell}-i_j)}{\prod_{j=1}^{r+1}(i_j+r+g-d)!}\Big]\theta^{N+g-\rho}.
\end{equation}
To prove the corresponding special maximal-rank locus $\cM^r_d(C)$ is non-empty, for any smooth projective curve $C$ of genus $g \leq d$, it suffices to show that
\begin{equation}\label{eqn:surj_coeff}
S(g,r,d):=(-1)^N\sum_{I:|I|\le\binom{r+1}{2}+N}\frac{2^{2N-|I|}d_I}{(N-\deg(s_I))!}\cd\frac{\prod_{1\le j<\ell\le r+1}(i_{\ell}-i_j)}{\prod_{j=1}^{r+1}(i_j+r+g-d)!}\neq 0.
\end{equation}
\begin{ex}\label{ex:ex1}
An important case is that in which $g=13$, $d=18$, and $r=6$. Here
\[
S(13,6,18)=-\sum_{I:|I|\le26}\frac{2^{10-|I|}d_I}{(26-|I|)!}\cd\frac{\prod_{1\le j<\ell\le 7}(i_{\ell}-i_j)}{\prod_{j=1}^7(i_j+1)!}.
\]
\end{ex}
\subsubsection{The injective range: $\binom{r+1}{2}<1+2d-g$.}
In this case, $\binom{r+2}{2}-m=1$ and we get
\[
S=[G^r_d(C)]\cd \Delta^{(1)}_{N'}(\SS)=(-1)^{(r+1)(d+1)+N'}s_J(\mU)\Big( \sum_{I:|I|\le\binom{r+1}{2}+N'}\frac{(4\theta)^{N'-\deg(s_I)}\psi_I}{(N'-\deg(s_I))!}s_I(\mU)\Big)
\]
where $N^{\pr}=2+2d-g-\binom{r+2}{2}$ and $J=(2g-1-d,2g-d,\hdots,2g-1-d+r)$. Consequently,
\[
q_*(S)=\sum_{I:|I|\le\binom{r+1}{2}+N'}\frac{(-4)^{N'-\deg(s_I)}\psi_I}{(N'-\deg(s_I))!}\cd\frac{\prod_{1\le j<\ell\le r+1}(i_{\ell}-i_j)}{\prod_{j=1}^{r+1}(i_j+r+g-d)!}\theta^{N'+g-\rho}.
\]
To show the corresponding SMRC locus is non-empty, it suffices to show that
\[
S^{\pr}(g,r,d):=\sum_{I:|I|\le\binom{r+1}{2}+N'}\frac{(-4)^{N'-\deg(s_I)}\psi_I}{(N'-\deg(s_I))!}\cd\frac{\prod_{1\le j<\ell\le r+1}(i_{\ell}-i_j)}{\prod_{j=1}^{r+1}(i_j+r+g-d)!}
\]
is nonzero. Although the non-emptiness of the SMRC locus in the injective range is not the focus of the current paper, we conclude this subsection by giving a concrete numerical example in this range. 
\begin{ex}\label{ex:ex2}
Consider the case where $g=6$, $d=8$, and $r=3$. Here
\[
S^{\pr}(6,3,8)=\sum_{I:|I|\le8}\frac{(-4)^{2-\deg(s_I)}\psi_I}{(2-\deg(s_I))!}\cd\frac{\prod_{1\le j<\ell\le 4}(i_{\ell}-i_j)}{\prod_{j=1}^{4}(i_j+1)!}
\]
and $q_*(S)$ is numerically equivalent to a zero-cycle of degree 10.
\end{ex}

\subsection{Relating the SMRC degree to special values of shifted Schur functions}\label{computing_d_I}
In order to certify the nonemptiness of SMRC loci, we need to show that the degree $S(g,r,d)$ described by Equation~\ref{eqn:surj_coeff} is non-zero. Interesting combinatorial questions arise as $S(g,r,d)$ involves the combinatorial coefficients $d_I$ of Theorem~\ref{thm:LLT1}(1).

\medskip
Recall from Equation~\ref{eqn:key} that each $d_I$ is a special value of a shifted Schur function.
In particular, setting $n=r+1$, we may rewrite the formula for the SMRC degree $S(g,r,d)$ in Equation \ref{eqn:surj_coeff} as follows:
\begin{equation}
S(g,r,d)=(-1)^N\sum_{I:|I|\le N+\binom{n}{2}}\frac{(-1)^{|\la(I)|}\cd 2^{2N-|\la(I)|}\cd s^*_{\la(I)}(\ep(r))\cd f^{\la(I)}\cd\prod_{j=0}^{r}j!}{(N-|\la(I)|)!\cd |\la(I)|!\cd\prod_{j=1}^{r+1}(i_j+r+g-d)!}
\end{equation}
where $N=\binom{r+2}{2}-(2d-g)\ge 0$, $\ep(r)=(r+1\hh 1)$ is the staircase partition of Definition~\ref{def:staircase_partition}, and $f^{\la(I)}=\frac{|\la(I)|!\prod_{1\le j<\ell\le r+1}(i_{\ell}-i_j)}{\prod_{j=1}^{r+1}(i_j)!}$ is the dimension of the irreducible representation of $\SS_{r+1}$ indexed by $\la(I)$ (see \cite[4.11]{fulton2013representation}).

\medskip
Next, we rewrite $S(g,r,d)$ in terms of evaluations of shifted Schur functions along staircases.
Indeed, by \cite[Thm 3.1]{okounkov1997shifted}, we have $s^*_{\la}(\ep(r))=0$ unless $\la\subset\ep(r)$. It follows that 
\begin{equation}\label{eqn:surj_coeff2}
S(g,r,d)={\prod_{\aaa=0}^r\aaa!\over\prod_{\aaa=0}^r (\aaa+r+g-d)!}\sum_{m=0}^N{(-1)^{N-m}2^{2N-m}\over (N-m)!\cd m!}\sum_{\substack{\la\subset\ep\\\la\vdash m}}\frac{s^*_{\la}(\ep(r))\cd f^{\la}}{\prod_{j=1}^{r+1}(\la_j+r+1-j+(r+g-d))_{\la_j}}.    
\end{equation}
\begin{rem}
Notice that the coefficient $\prod_{\aaa=0}^r\aaa!\over\prod_{\aaa=0}^r (\aaa+r+g-d)!$ is the coefficient of the cohomology class of $W^r_d(C)$ (see Example~\ref{ex:Wrd}).
\end{rem}

Hereafter, let $\ov{S}(g,r,d)$ be the ``reduced" version of $S(g,r,d)$ defined by $S(g,r,d)=\frac{\prod_{\aaa=0}^r\aaa!}{\prod_{\aaa=0}^r(\aaa+r+g-d)!}\cd\ov{S}(g,r,d)$. We shall introduce one last auxiliary notion in order to simplify the computation of $\ov{S}(g,r,d)$ in the next subsection.
\begin{defn}
Given positive integers $g,r$ and $d$, let
\[F_{g,r,d}(m):={2^{2N-m}\over(N-m)!m!}\sum_{\la\vdash m}\frac{ F_{\la}(r)f^{\la}}{\prod_{j=1}^{r+1}(\la_j+(r+1-j)+(r+g-d))_{\la_j}}={2^{2N-m}\over(N-m)!m!}\sum_{\la\vdash m}\frac{ F_{\la}(r)f^{\la}}{(A\upharpoonright\la)}\]
where $A=2r+1+g-d$ and $(A\upharpoonright\la)$ is the generalized raising factorial (see Notation~\ref{nt:raising_fac}).
\end{defn}
Clearly $S(g,r,d)$ is nonzero if and only if $\ov{S}(g,r,d)=\sum_{m=0}^N(-1)^{N-m}F_{g,r,d}(m)$ is nonzero. Note that the class calculation carried out for $W^r_d(C)$ in Example~\ref{ex:Wrd} implies that the class $q_*(S)$ of equation~\eqref{eqn:surj_range} is precisely
\begin{equation}\label{eq:qS}
q_*(S)=[\sum_{m=0}^N(-1)^{N-m}F_{g,r,d}(m)]\cd[W^r_d(C)]\cd\theta^N.
\end{equation}

\begin{rem}
Using the standard (Hall) scalar product $\lag\cd,\cd\rag$ on the ring of symmetric functions $\Lambda$, we may rewrite $\ov{S}(g,r,d)$ in a more compact form. Indeed, the Schur functions determine an orthonormal basis for $\Lambda$ with respect to $\lag\cd,\cd\rag$, while the Murnaghan--Nakayama rule \cite[Thm 7.17.4]{stanley_fomin_1999} implies that $p_{1^m}= \sum_{\la \vdash m} f^{\la} s_{\la}$, where $p_{1^m}= \sum_j x_j^m$ denotes the $m$-th power sum symmetric function. It follows that
\begin{equation}
\ov{S}(g,r,d)=\bigg \lag\sum_{\la}\frac{ F_{\la}(r)s_{\la}}{(A\upharpoonright\la)} ,\sum_{m=0}^N{(-1)^{N-m}2^{2N-m}\over(N-m)!m!}p_{1^m}\bigg \rag={2^N\over N!}\bigg \lag\sum_{\la}\frac{ F_{\la}(r)s_{\la}}{(A\upharpoonright\la)},(p_{1}-2)^N \bigg \rag.
\end{equation}
\end{rem}

\subsection{$F_{g,r,d}(m)$ for small values of $m$}\label{subsec:F_g,r,d(m)_for_small_m}
We now apply the technical results of the final part of Subsection~\ref{computing_d_I} to compute $F_{g,r,d}(m)$ whenever $m \leq 7$.
\begin{lem}\label{lem:m_small_value} Set $A=2r+1+g-d$ and $B=\binom{r+2}{2}$. We have
\begin{enumerate}
\item $F_{g,r,d}(0)={2^{2N}\over N!}$;
\item $F_{g,r,d}(1)=\frac{2^{2N-1}B}{A(N-1)!1!}$;
\item $F_{g,r,d}(2)=\frac{2^{2N-2}B_2A}{(A+1)_3(N-2)!2!}$;
\item $F_{g,r,d}(3)=\frac{2^{2N-3}}{(A+2)_5(N-3)!3!}[B_3(A^2-2)-2B_2]$;
\item $F_{g,r,d}(4)=\frac{2^{2N-4}}{A(A+3)_7(N-4)!4!}[B_4(A^2-1)(A^2-9)+B_2(B-3)(B-6)(2A^2-3)]$;
 \item $F_{g,r,d}(5)=\frac{2^{2N-5}}{(A+4)_9(N-5)! 5!}\cdot  \big[B(B-1)(B-3)(72 + A^4 (B-4) (B-2) - 20 A^2 (B-4) (B-1) + 6 B (13B-48))\big]$;
    \item $F_{g,r,d}(6)= \frac{2^{2N-6}}{(A+5)_{11}(A+1)_3(N-6)! 6!}\cdot \big[B(B-1)(B-3)((B^3 - 11B^2 + 38B - 40)A^8 - (41B^3 - 411B^2 + 1198B - 840)A^6 + 2(229B^3 - 2009B^2 + 4636B - 1680)A^4 - 2(629B^3 - 4629B^2 + 8256B - 1280)A^2 + 240B^3 - 1440B^2 + 4800B)\big]$; and
    \item $F_{g,r,d}(7)= \frac{2^{2N-7}}{(A+6)_{13} (A^2-1)(N-7)! 7!}\cdot \big[B(B-1)(B-3)(B-6)((B^3 - 11B^2 + 38B - 40)A^8 - (71B^3 - 711B^2 + 2068B - 1440)A^6 + 14(112B^3 - 977B^2 + 2233B - 840)A^4 - 2(5699B^3 - 40149B^2 + 67266B - 9680)A^2 + 15780B^3 - 73200B^2 + 105300B - 9000)\big]$.
\end{enumerate}
\end{lem}
\begin{proof}
When $m\le 2$, the results follow directly from the definition of $F_{g,r,d}(m)$.
Now say $m=3$. Theorem~\ref{thm:ok1} and Corollary~\ref{cor:prob} imply that $2F_{(2,1)}(r)=\binom{r+2}{2}_3-2F_{(3)}(r)=B_3-2F_{(3)}(r)$, and it follows that 
\[
\begin{split}
    F_{g,r,d}(3)&=\frac{2^{2N-3}}{(N-3)!3!}\bigg(\frac{F_{(3)}(r)}{(A+2)_3}+\frac{B_3-2F_{(3)}(r)}{(A+1)_3}+\frac{F_{(1,1,1)}(r)}{(A)_3}\bigg)\\
    &=\frac{2^{2N-3}}{(A+2)_5(N-3)!3!}(B_3(A^2-2)-2B_2).
\end{split}
\]
The determination of $F_{g,r,d}(m)$ for $4 \leq m \leq 7$ is similar, but more involved. We apply Proposition~\ref{prop:polynomiality}, Lemma~\ref{OOthm9.1} and Lemma~\ref{OOthm11.1} to compute all of the functions $F_{\lambda}(r)$ with $|\lambda| \leq 7$. In doing so, we use MATLAB to explicitly solve a number of linear systems of equations.
\end{proof}
In Lemma \ref{lem:m_small_value} we computed the functions $F_{g,r,d}(m)$ explicitly for $m\le 7$. Since $\ov{S}(g,r,d)$ is an alternating sum of all $F_{g,r,d}(m)$ with $m\le N$, we thus obtain explicit formulas for $\ov{S}(g,r,d)$ whenever $N\leq 7$. When $N$ is small, we
can explicitly check the positivity of $\ov{S}(g,r,d)$ (and hence the positivity of $S(g,r,d)$). A key point is that whenever $m \le 7$, all $F_{g,r,d}(m)$ are expresed in terms of $A=2r+1+g-d$ and $B=\binom{r+2}{2}$; it turns out to be useful to estimate the ratio $\frac{B}{A}$.
\begin{lem}\label{lem:estimate}
Let $p=\frac{B}{A}\in\qq$. Under the numerical assumptions of Basic Inequality~\ref{lem:num_assumption}, we have 
\begin{enumerate}
    \item $p\ge 2$; and if $p=2$, then $(g,r,d) \in \{(1,2,3),(5,3,7)\}$. 
    \item $p\ge\frac{r+2}{3}$.
\end{enumerate}
\end{lem}
\begin{proof}
The fact that $r \geq 2$ for every $(g,r,d)$ in the surjective range is clear. Moreover, we have
\[B=\binom{r+2}{2}=2d-g+N=g-2(r+g-d)+2(r+1)-2+N=2A+g-4(r+g-d)-2+N.\]
Whenever $r\ge 3$, we have $g-4(r+g-d)-N\ge D(g,r,d)$, and therefore $B-2A\ge 2N-2\ge 0$, so that $p\ge 2$. On the other hand, when $r=2$, Basic Inequality~\ref{lem:num_assumption} implies that
\[
1+3(2+g-d)\le g\le \binom{r+2}{2}-N=2d-g= 6-N\le 5
\]
and hence that $2+g-d \in \{0,1\}$. If $2+g-d=0$, the only possibility is that $(g,r,d)=(1,2,3)$, in which case $B-2A=0$, i.e. $p=2$; while if $2+g-d=1$, no triple $(g,r,d)$ satisfies our numerical conditions.  

\medskip
The only other situation in which $p-2$ may not be strictly positive is when $r\ge 3$ and $N=1$. When this happens, $p-2$ will be strictly positive as soon as $r+g-d>0$ and $r>3$, as it then follows that $g-4(r+g-d)-N>D(g,r,d)$. It therefore remains to check those cases in which $r+g-d=0$ or $r=3$. 

\medskip
Accordingly, suppose that $r+g-d=0$ and $N=1$. If $B-2A=g-1$ vanishes, then necessarily $g=1$ and $r=d-1$. In that case, the fact that $N=\binom{d+1}{2}-2d+1=1$ implies that $d=3$ and $r=2$. On the other hand, if $r=3$, then $g-4(r+g-d)-1=0$ and $\binom{r+2}{2}=2d-g=9$ together force $g=5$ and $d=7$. So item (1) is proved.

\medskip
To verify item (2), note first that Basic Inequality~\ref{lem:num_assumption} also implies that 
\[
\binom{r+2}{2}-(r+1)(r+g-d)\ge 2N
\]
or equivalently, that
\[(r+2)-2(r+g-d)\ge\frac{4N}{r+1}>0.\]
It follows that $r+1\ge 2(r+g-d)$, i.e., that $A\le \frac{3}{2}(r+1)$. So we deduce that $p\ge\frac{\binom{r+2}{2}}{\frac{3}{2}(r+1)}=\frac{r+2}{3}$. 
\end{proof}
\medskip

\begin{ex}\label{ex:N_small}
Let $p=\frac{B}{A}$ as above. When $N=1$, we have
\begin{equation}
\ov{S}(g,r,d)=\frac{2B}{A}-4=2p-4.
\end{equation}
Applying Lemma \ref{lem:estimate}(1), we deduce that $\ov{S}(g,r,d)$ is non-negative, and strictly positive except when $(g,r,d) \in\{(1,2,3),(5,3,7)\}$. 

\medskip
Similarly, when $N=2$, Basic Inequality~\ref{lem:num_assumption} implies that $r\ge 3$, while
\[
\ov{S}(g,r,d)=\frac{2B_2A}{(A+1)_3}-\frac{8B}{A}+8>2 \bigg[\bigg(\frac{B-1}{B}\bigg)p^2-4p+4\bigg].
\]
Applying Lemma~\ref{lem:estimate}(2), we see that whenever $r\ge 6$, 
\[
p\ge\frac{8}{3}>\bigg(2+2\sqrt{\frac{1}{B}}\bigg)\frac{B}{B-1}
\]
in which the rightmost quantity is the larger root of the polynomial $f(x)=(\frac{B-1}{B})x^2-4x+4$. It follows that $\ov{S}(g,r,d)>0$ in this case. 

\medskip
We are left to deal with those cases in which $r \in \{3,4,5\}$. When $r=5$, from Basic Inequality~\ref{lem:num_assumption} we have $2+6(r+g-d)\le 19$ and thus $r+g-d\le 2$. Lemma~\ref{lem:estimate}(2) now yields $p\ge\frac{21}{6+2}>(2+2\sqrt{\frac{1}{B}})\frac{B}{B-1}$. Applying the argument of the preceding paragraph, we deduce that $\ov{S}(g,r,d)>0$. 

\medskip
When $r=4$, 
the unique triple $(g,r,d)$ that fulfills all of our numerical constraints is $(g,r,d)=(7,4,10)$, in which case $\ov{S}(7,4,10)=0$. Finally, $r=3$, there is a unique possibility, namely $(g,r,d)=(2,3,5)$; in this case, $\ov{S}(2,3,5)=0$ as well. 
\end{ex}

Calculating $F_{g,r,d}(m)$ explicitly becomes difficult as soon as $m\ge 8$. However, by applying a single branching rule (Lemma \ref{OOthm9.1}) for shifted Schur functions, we may conclude that $q_*(S)$ is a positive class when $r$ is relatively large with respect to $N$. To this end, we introduce the following rational function on partitions. 
\begin{nt}\label{nt:h_func}
Given any partition $\la$ of length at most $r+1$, let 
\[h_{\la}(g,r,d):=\prod_{j=1}^{r+1}((2r+1+g-d)+\la_j-j)^{-1}_{\la_j}.\] 
\end{nt}
Notice that the rational function $h_{\la}(g,r,d)$ is precisely the reciprocal of the product of all the \textit{shifted contents} $A+j-i$, for $(i,j)\in D(\la)$. The following result is straightforward to verify.

\begin{lem}
We have
\[h_{\la}(g,r,d)=[\la_j+1-j+(2r+1+g-d)]h_{\la^+}(g,r,d)\]
for all partitions $\la$ and $\la^+$ related by $\la^+-\la=(0\hh 0,\underbrace{1}_{j-\text{th}},0\hh 0)$. 
\end{lem}

\begin{prop}\label{prop:asymp}
Fix a choice of $N=\binom{r+2}{2}-2d+g\ge 1$. The SMRC degree $S(g,r,d)$ is a positive rational number whenever $r\ge 12N-2$.
\end{prop}

\begin{proof}
Let $a_{g,r,d}(m):={2^{2N-m}\binom{\binom{r+2}{2}}{m}\over(N-m)!}$, where $N=\binom{r+2}{2}-2d+g\ge1$, whenever $m\le N$. Whenever $m<N$, we further set
\[
c_{m+1}(r):=\frac{a_{g,r,d}(m+1)}{a_{g,r,d}(m)}={N-m\over 2}\cd\frac{\binom{r+2}{2}-m}{m+1}.
\]
For any fixed value of $N$, $c_{m+1}(r)$ is a quadratic polynomial in $r$ with highest-degree coefficient equal to ${N-m\over 4(m+1)}>0$.

{\flushleft Given} a partition $\la$, let $r_{\la}:=\frac{f^{\la} s_{\la}^{\ast}(\ep(r))}{\binom{r+2}{2}_{|\la|}}$. According to Lemma~\ref{OOthm9.1}, we have  
\[
\sum_{\substack{\la\vdash m\\\la\subset\ep}}r_{\la}h_{\la}=\sum_{\la}\Big(\sum_{\la^+:\la^+\subset\ep}[\la_j+1-j+(2r+1+g-d)]\cd \bigg(\frac{f^{\la}}{f^{\la^+}} \bigg)\cd r_{\la^+}h_{\la^+}\Big)
\]
where $\la^+-\la=(0\hh 0,\underbrace{1}_{j-\text{th}},0\hh 0)$. Equivalently, we have  
\begin{equation}\label{eq:branch}
\sum_{\substack{\la\vdash m\\\la\subset\ep}}r_{\la}h_{\la}=\sum_{\substack{\mu\vdash m+1\\\mu\subset\ep}}\Big(\sum_{\mu^-}[\mu_j-j+(2r+1+g-d)]\cd \Big(\frac{f^{\mu^-}}{f^{\mu}}\Big)\cd r_{\mu}h_{\mu}\Big)
\end{equation}
where 
$\mu-\mu^-=(0\hh 0,\underbrace{1}_{j-\text{th}},0\hh 0)$.

\medskip
Now say $N=2p-1$ is odd. Define $A_i:=-F_{g,r,d}(2i-2)+F_{g,r,d}(2i-1)$, for $i=1\hh p$. $\ov{S}(g,r,d)$ then decomposes as $\ov{S}(g,r,d)=A_1+..+A_p$, and equation~\eqref{eq:branch} implies that
\[
A_i=a_{g,r,d}(2i-2)\cd\sum_{\substack{\la\vdash 2i-1\\\la\subset\ep}}r_{\la}\Big(c_{2i-1}(r)-\sum_{\la^-}[\la_j-j+(2r+1+g-d)]\cd \Big(\frac{f^{\la^-}}{f^{\la}}\Big)\Big)h_{\la}(g,r,d).
\]
Here $r_{\la}\in\QQ\cap(0,1]$, while the quotients $\frac{f^{\la^-}}{f^{\la}}$ are positive rational numbers such that $\sum_{\la^-}\frac{f^{\la^-}}{f^{\la}}=1$, by the usual branching rule. More importantly, we have
\[\la_j-j+(2r+1+g-d)\le 3r+1\] 
since $\la\subset \ep$ and $d\ge g$. Therefore, provided $c_{2i-1}(r)> 3r+1$ holds for all $i$, we get $S(g,r,d)>0$. Since $c_{2i-1}(r)\ge c_{N}(r)$, it further suffices to establish that $c_N(r)> 3r+1$. Now $c_N(r)$ is a quadratic polynomial in $r$, with positive highest-degree coefficient, it is easy to see that $c_N(r)> 3r+1$ when $r$ is sufficiently large with respect to $N$. 

\medskip
Similarly, when $N=2p$ is even, we decompose $\ov{S}(g,r,d)$ as $B_0+B_1+...+B_p$, where $B_0=F_{g,r,d}(0)$ and $B_i=-F_{g,r,d}(2i-1)+F_{g,r,d}(2i)$ for $i\ge 1$. An argument analogous to that used in the case of odd $N$ allows us to conclude that $c_N(r)> 3r+1$ when $r$ is sufficiently large with respect to $N$. 

\medskip
On the other hand, $c_N(r)>3r+1$ is equivalent to the statement that the quadratic polynomial $Q_N(r):=r^2 + r(3-12N)+4-6N$ is positive. It is not hard to see that the largest root of $Q_N(r)$ in $r$ is less than $12N-2$.
\end{proof}

Finally, we classify all admissible triples $(g,r,d)$ according to the positivity of their SMRC loci whenever $N \leq 7$.

\begin{prop}\label{prop:positivity_for_small_N}
The class $q_*(S)$ of equation~\eqref{eqn:surj_range} is strictly positive when $N \leq 2$
except when either
\begin{itemize}
    \item[(i)] $N=1$, and $(g,r,d) \in \{(1,2,3), (5,3,7)\}$; or
\item[(ii)] $N=2$, and $(g,r,d) \in \{(2,3,5), (7,4,10)\}$
\end{itemize}
and $q_*(S)$ is unconditionally strictly positive whenever $3\leq N \leq 7$.
\end{prop}

\begin{proof}
Given Example \ref{ex:N_small}, it only remains to show that $\ov{S}(g,r,d)>0$ for $N$ between 3 and 7. 

\medskip
Applying Proposition~\ref{prop:asymp}, we conclude that $\ov{S}(g,r,d)>0$ whenever $r \geq 12N-2$. Thus we are left over with finitely many cases of feasible triples $(g,r,d)$. We then determine a lower bound for $r$ for each value of $N$ by applying Basic Inequality~\ref{lem:num_assumption}: 
\begin{enumerate}
    \item When $N=3$, $r\ge 4$;
    \item when $N=4,5$, $r\ge 5$;
    \item when $N=6,7$, $r\ge 6$.
\end{enumerate}

Finally, we check the positivity for the finitely many remaining possible triples $(g,r,d)$ using MATLAB; see the ancillary file:

\medskip
\url{https://drive.google.com/file/d/1Az5WOZyoa_UzQvktT7KuilTpbddL4ANn/view}. 

\medskip
In this file, we list those values of $\ov{S}(g,r,d)$ all $(g,r,d)$ for which $3\le N\le 7$ and $r$ lies between the given lower bound and $12N-3$.  
\end{proof}
\begin{rem}
The fact that $S(1,2,3)=S(2,3,5)=0$ is consistent with Mumford's theorem \cite[Thm 6]{mumford2010varieties}, which establishes that the quadratic multiplication map $v_2$ associated with (the complete series determined by) a line bundle $L$ with $\deg(L)\ge 2g+1$ is surjective.
\end{rem}



\section{The Bertram-Feinberg-Mukai conjecture and its connection with the SMRC}\label{sec:bfm}
In this section, we explain the connection between the BFM conjecture (Conjecture \ref{BFM}) and the Strong Maximal Rank Conjecture for quadrics. 

\medskip
Our point of departure is the fact that every stable rank-two vector bundle $E$ with canonical determinant fits into a short exact sequence of the form
\begin{equation}\label{eqn:ext}
    0\to\omega\ot L^{-1}\to E\to L\to 0.
\end{equation}
Every extension \eqref{eqn:ext} naturally determines an element $e \in  \ext^1(L,\omega\ot L^{-1})\cong H^0(L^{\ot 2})^{\vee}$. Given any such extension $e$, we have $h^0(E)=h^1(L)+\dim \ker(u_e)$, where $u_e$ is the linear map $H^0(L)\to H^1(\omega\ot L^{-1})\cong H^0(L)^{\vee}$ in the cohomological long exact sequence induced by \eqref{eqn:ext}.

\medskip
In order to verify (the existence portion of) Conjecture~\ref{BFM}, our aim is to produce extensions \eqref{eqn:ext} of line bundles $[L]\in\Pic^d(C)$ whose associated rank-two vector bundles $E$ are stable and satisfy $h^0(E)\ge k$, whenever $k(k+1)\le 6g-6$.
In doing so, we try to simultaneously ensure that $d$ is as small as possible relative to $k$ and that $u_e$ has small rank. Heuristically speaking, we specify a stable vector bundle by identifying one of its \textit{minimal} quotient line bundles, to the extent that this is possible given the cohomological condition $h^0(E)\ge k$. 

\medskip 
The impulse to consider extensions of (relatively) small degree line bundles comes from two sources. First of all, maximal sub-line bundles (and hence minimal quotient bundles) of a vector bundle are in some sense canonical. For example, it is known that the space of maximal subbundles of any rank-two vector bundle is at most one-dimensional \cite[Cor. 4.6]{msub}, and is further known to be finite or even a singleton \cite[Cor. 3.2 and Prop. 3.3]{msub} for general rank-two vector bundles with certain prescribed degrees and Segre invariants. Second, it is easier to certify that rank-two bundles determined by extensions of line bundles of relatively small degrees are stable. In fact, we shall see that in many such cases stability is automatic.

\medskip
On the other hand, since we focus on extensions $e$ of line bundles in which $u_e$ has relatively large {\it cokernel} dimension, quite often the loci of the line bundles being extended are related to the SMRC loci for quadrics we defined in Section~\ref{smrc}. To see this, we analyze the two respective conditions:
\begin{enumerate}
\item $h^1(L)\ge p$; and
\item $\dim\ker(u_e)\ge k-p$.
\end{enumerate}

The locus of line bundles $L$ in $\Pic^d(C)$ with $h^1(L)\ge p$ is easy to describe: it is precisely the Brill-Noether locus $W^{p+d-g}_d(C)$.

\medskip
We now turn to the dimension of $\ker(u_e)$. The assignment that takes an extension $e$ to the corresponding coboundary map $u_e$ describes a linear map from $\ext^1(L,\omega_C\ot L^{-1})\cong H^0(L^{\ot 2})^{\vee}$ to $\h(H^0(L),H^0(L)^{\vee})\cong H^0(L)^{\vee}\ot H^0(L)^{\vee}$, which is dual to the multiplication map $\mu_L:H^0(L)^{\ot 2}\to H^0(L^{\ot 2})$. It is easy to see that the condition $\dim\ker(u_e)\ge k-p$ may be reformulated as the statement that some $(k-p)$-dimensional subspace $V\sbs H^0(L)$ is such that $\im(\mu_{L,V})\sbs\ker(e)$, where $\mu_{L,V}$ is the restriction of $\mu_L$ to $H^0(L)\ot V$, and $e$ is viewed as a linear function on $H^0(L^2)$. (See \cite[Remark 5.7]{ext2} for a more general statement.)

\medskip
The upshot is that in order to show existence of rank two linear series with canonical determinant and many sections, it is useful to study loci in $\ext^1(L,\omega\ot L^{-1})$ of specified rank.

\begin{defn}
Let $\cW_t$ denote the locus of $e$ in $\ext^1(L,\omega\ot L^{-1})$ where $\dim\rk(u_e)\le h^0(L)-t$.
\end{defn}
It is not hard to see that $\cW_t$ is a determinantal scheme of expected codimension $t^2$. Indeed, it is precisely the $((r+1)-t)$-th degeneracy locus of the pull-back of the universal linear map on $\h(H^0(L),H^0(L)^{\vee})$. Accordingly, we have a natural filtration
\[\ext^1(L,\omega\ot L^{-1})\supset\cW_1\supset\cW_2\supset\hdots\supset\cW_{r+1}.\]
Note that $\cW_{r+1}=0$ if and only if the multiplication map $\mu_L$ is surjective.

For our application to higher-rank Brill--Noether theory, we will mainly consider \textit{non-general} extensions of (possibly) non-general linear series by their Serre duals. Whether or not invertible subsheaves of $L$ lift to subsheaves of (stable) rank two bundles $E$ is related to the existence of $D$ on $C$ that fail to impose independent conditions on global sections of $L$. 
\begin{prop}\label{prop:sec}
Suppose $e\in \ext^1(L,\omega\ot L^{-1})$ is nonzero, where $\deg(L)=g+a$ with $a>0$. Let $E$ denote the rank-two vector bundle obtained from $e$, and suppose further that $h^0(E)=h^0(L)+h^1(L)$. Suppose $Z$ is an effective divisor on $C$.
\begin{enumerate}
    \item Whenever there exists a morphism $i:L(-Z)\to E$ for which the composition $L(-Z)\stackrel{i}{\to} E\to L$ is the identity and realizes $L(-Z)$ as a sub-bundle of $E$, we have 
    \[\deg(Z)\ge 2\dim(H^0(L)/H^0(L(-Z))).\]
    \item If, moreover, $\deg(Z) \le a/2$ and $\mu_{L(-Z)}$ is surjective, then $E$ is not stable.
\end{enumerate}

\end{prop}
\begin{proof}
The assumption in item one implies there is a short exact sequence 
\[0\to L(-Z)\to E\to \omega_C\ot L^{-1}(Z)\to 0\]
and hence $h^0(L(-Z))+h^1(L(-Z)) \geq h^0(L)+h^1(L)=h^0(E)$. On the other hand, Riemann--Roch together with the long exact sequence in cohomology induced by the inclusion of $L(-Z)$ in $L$ imply that
\[
h^0(L(-Z))+h^1(L(-Z))=h^0(L)+h^1(L)+\deg(Z)-2\dim(H^0(L)/H^0(L(-Z)))
\]
and the claim in item one follows immediately.

Now suppose $\deg(Z)\le a/2$ and $\mu_{L(-Z)}$ is surjective. The multiplication map $\mu_L$ factors as 
\[H^0(L)^{\ot 2}\twoheadrightarrow \im(\mu_L)\hookrightarrow H^0(L^{\ot 2}).\]
The fact that $h^0(E)=h^0(L)+h^1(L)$ implies that $e$ is in the kernel of $H^0(L^2)^{\vee}\to \im(\mu_L)^{\vee}$; since $\im(\mu_L)$ contains $\im(\mu_{L(-Z)})=H^0(L^2(-2Z))$, $e$ vanishes on $H^0(L^2(-2Z))$. It follows that $e$ is in the kernel of the dual of the inclusion map $H^0(L^2(-2Z))\to H^0(L^2)$, namely
\[
\ext^1(L,\omega\ot L^{-1}) \to \ext^1(L(-2Z),\omega\ot L^{-1}).
\]
This means, in turn, that the top row of the following diagram splits:
\[\begin{tikzcd}
0\ar[r]& \omega_C\ot L^{-1}\ar[r]\ar[d]&E^{\pr}\ar[r]\ar[d]& L(-2Z)\ar[r]\ar[d]&0\\
0\ar[r]& \omega_C\ot L^{-1}\ar[r]&E\ar[r]& L\ar[r]&0
\end{tikzcd}\]
where $E^{\pr}= E \times_L L(-Z)$. Thus $L(-2Z)$ is an invertible sub-sheaf of $E^{\pr}$ (and hence of $E$). Since $\deg(Z)\le a/2$, we have $\deg(L(-2Z))\ge g$, and consequently $E$ is not stable.
\end{proof}

\begin{rem}
Geometrically, Proposition~\ref{prop:sec} relates the liftability of invertible subsheaves of $L$ to the existence of divisors $Z$ whose images under $|L|$ span {\it secant spaces} of dimensions prescribed by the rank of the evaluation map $H^0(L) \to H^0(L|_Z)$. This, in turn, explains our terminological use of ``secant divisors".
\end{rem}

Finally, a classical result of Nagata gives a lower bound on the degree of a maximal subbundle of a rank-two vector bundle:  
\begin{thm}[\cite{Nag}]\label{thm:nagata}
Any maximal-degree line subbundle $F$ of a rank-two vector bundle $E$ over a non-singular projective curve of genus $g$ over an algebraically closed field satisfies
\[
\deg(E)- 2\deg(F) \leq g.
\]
\end{thm}
Applying this theorem to the case where $E$ is a rank-two vector bundle with canonical determinant yields the following useful statement.
\begin{cor}\label{cor:min_quot}
The degree of the minimal quotient line bundle of a rank-two vector bundle with canonical determinant is at most $\lfloor{3g\over 2} \rfloor-1$.
\end{cor}
So by replacing a line bundle $L$ with its Serre dual if necessary, it suffices to consider $L$ for which $g-1\le \deg(L)\le \lfloor{3g\over 2} \rfloor-1$ when constructing stable rank-two vector bundles with canonical determinant via extensions as in \eqref{eqn:ext}.

\subsection{The search for minimal quotient line bundles}
In order to prove the existence portion of the BFM conjecture it suffices to show that for every $g\ge 2$ and for the \textit{maximal} integer $k=k(g)$ for which $\rho_{g,k}\ge 0$, there exists a stable rank-two vector bundle $E$ over a general genus $g$ curve for which $\det(E)=\omega$ and $h^0(E)\ge k$. 

\medskip
Recall that we are interested in extensions of the form \eqref{eqn:ext} over a curve $C$ that is (Brill-Noether and Petri-) general. 
We would also like to minimize the degree of $L$ in such extensions. In light of this, Corollary~\ref{cor:min_quot} motivates the following definition.
\begin{defn}\label{defn:min_deg}
The {\it minimal BN-compatible degree} with respect to $k$ is
\[
d^*_k:=\min\{d:g-1\le d\le \lfloor{3g\over 2}\rfloor-1\mid \exists L \in \Pic^d(C) \text{ such that } h^0(L)+h^1(L)\ge k\}.
\]
\end{defn}
Now suppose $\deg(L)=d^*_k$. The Brill-Noether theorem and Serre duality 
together imply that
\begin{equation}\label{eq:BN}
(r+1)(r+g-d)\le g \text{ and } (r+1)+(r+g-d)\ge k
\end{equation}  
where $h^0(L)=r+1$, $\deg(L)=d$ and $h^1(L)=r+g-d$. 
\begin{defn}
Fix positive integers $(g,k)$ such that $\rho_{g,k}\ge 0$. Suppose $g-1\le d\le \lfloor{3g\over 2}\rfloor-1$ and that $(r,d)$ satisfies the constraints in \eqref{eq:BN}; we then say that $(r,d)$ is \textit{BN-compatible} with respect to $(g,k)$.
\end{defn}

\medskip
It is not always the case that a stable rank-two vector bundle $E$ for which $\det(E)\cong \omega_C$ and $h^0(E)=k$ has a quotient line bundle of degree $d^*_k$. For example, if for all line bundles $L$ of degree $d^*_k$, the multiplication map $\mu_L$ is surjective and $h^0(L)+h^1(L)=k$, then every extension of the form \[0\to \omega\ot L^{-1}\to E\to L\to 0\] 
splits (and hence $E$ is not stable). However, on the positive side we have the following result.
\begin{prop}\label{prop:basic_case}
Suppose $e:0\to \omega_C\ot L^{-1}\to E\to L\to 0$ is a non-trivial extension of line bundles such that $h^0(E)\ge k$ and $\deg(L)=d^*_k \geq g$, then $E$ is stable and $L$ is a minimal quotient line bundle of $E$.
\end{prop}
\begin{proof}
Let $0\to F\to E$ be a sub-line bundle of $E$. The composition of morphisms of line bundles $F\to E\to L$ is either zero or injective. If it is zero, the morphism $F\to E$ must factor through the kernel of $E\to L$, which is $\omega_C\ot L^{-1}$. Then $\deg(F)\le 2g-2-\deg(L)\le g-2$. 

\medskip
If it is injective, let $d^{\pr}=\deg(F)\le\deg(L)=d^*_k$. Since $h^0(F)+h^1(F)\ge k$ by construction, we must have either $\deg(F)<g-1$ or $\deg(F)=\deg(L)$. However, the second case is impossible as this would force $F\cong L$ and thereby violate our assumption that $e$ is a non-trivial extension. 

\medskip
It follows that $E$ admits no sub-line bundle of degree $g-1$ or greater. Hence $E$ is stable. The fact that $L$ is minimal follows from the definition of $d^*_k$ and the stability of $E$.
\end{proof}
\begin{cor}\label{cor:tri}
The existence portion of the BFM conjecture holds under either of the following two circumstances.
\begin{enumerate}
    \item There exists a BN-compatible pair $(r,d^*_k)$ such that $2r+1+g-d^*_k=k$, $d^*_k\ge g$ and $(r,d^*_k)$ falls within the injective range. 
    \item There exists a BN-compatible pair $(r,d^*_k)$ such that $2r+1+g-d^*_k=k$, $d^*_k\ge g$, $(r,d^*_k)$ falls within the surjective range, and $\mathcal{M}^r_{d^*_k}(C)\neq\emptyset$. 
\end{enumerate}
\end{cor}
\begin{proof}

In both cases, the multiplication map $\mu_L$ fails to be surjective and hence its dual $\mu^{\vee}_L$ fails to be injective. Case (1) follows from the classical Maximal Rank Conjecture, which is now a theorem of Eric Larson; see \cite{larson2017maximal}. Larson's theorem implies that there exists a nonzero extension $e$ in $\ext^1(L,\omega_C\ot L^{-1})$ with $u_e=0$, in which case $\dim\ker(u_e)=h^0(L)$ and hence $h^0(E)=k$.
In both cases, the stability of $E$ follows from Proposition~\ref{prop:basic_case}.
\end{proof}
\begin{rem}
Note that we ignore the possibility that $d^*_k=g-1$ for a practical reason. Namely, when $g\ge 13$ and $k$ is relatively large with respect to $g$, the condition $d^*_k\ge g$ will always hold as there is no strictly semi-stable rank-two vector bundle on a Brill-Noether general curve with $k$ independent sections. See \cite[Lemma 3.12]{Zhang1} for a precise statement. 
\end{rem}
\begin{ex}
The following are two examples of those cases covered by Corollary~\ref{cor:tri}.

\begin{enumerate}
    \item Set $(g,k)=(14,8)$. In this case, $d^*_8=17$, $(5,17)$ is a BN-compatible pair that falls within the surjective range, and $N(g,r,d)=1$. It follows from our calculation of the SMRC class in Proposition \ref{prop:positivity_for_small_N} that $\mathcal{M}^5_{17}(C)$ is nonempty. Consequently, the existence portion of the BFM conjecture holds for $(g,k)=(14,8)$.  
    \item Set $(g,k)=(18,9)$. In this case, $d^*_9=20$, and $(5,20)$ is a BN-compatible pair that falls within the {\it injective} range. Thus, the existence portion of the BFM conjecture also holds for $(g,k)=(18,9)$.  
\end{enumerate}
\end{ex}

\medskip
 


\subsection{Existence in small genera}\label{subsubsec: small_genera} The existence portion of the BFM conjecture for small genera ($g\le 12$) was first established by Bertram and Feinberg in \cite{Ber}. Here, we show how these cases of the conjecture can be easily recovered from our MRC-based viewpoint.


\medskip
The $k$-values listed here are maximal with respect to the given $g$-values such that $\rho_{g,k}=3g-3-\binom{k+1}{2}\ge0$. As we shall see, the case $g=2$ is exotic in the sense that $\rho_{2,2}=0$, but there is no stable bundle of rank two with two sections on a genus 2 curve. We nevertheless describe the situation in this case, since the approach is the same as for other low genera cases. 

\medskip
\noindent $\mathbf{g=2, k=2.}$ We will show that every semi-stable rank-two vector bundle with canonical determinant and two sections is strictly semi-stable. The minimal quotient line bundle of a rank two vector bundle $E$ with canonical determinant on a genus 2 curve $C$ has degree at most $\lfloor{3g\over 2}\rfloor-1=2$. For $h^0(E)\ge2$ to hold, $E$ must fit into an extension $e$ of the form 
\[0\to \mO_C\to E\to\omega_C\to 0\]
since $h^0(L)+h^1(L)\ge 2$ and $\deg(L)=2$ together imply that $L=\omega_C$. 

\medskip
Now consider $e$ as a point in $\PP(H^0(\omega_C^{\ot 2})^{\vee})$. By \cite[Prop. 1.1]{msub}, $e$ lies in $\Sec_1(X)=X$ whenever $E$ is {\it not} stable, where $X$ is the image of the morphism $C\to\PP(H^0(\omega_C^{\ot 2})^{\vee})$. In particular, a general extension $e$ will contribute a stable bundle $E$. On the other hand, it is easy to check directly that $\mu_{\omega_C}:\sy^2 H^0(\omega_C)\to H^0(\omega_C^{\ot 2})$ is an isomorphism, since $h^0(\omega_C)=2$. It follows that the dual $\mu_{\omega_C}^{\vee}$ is also an isomorphism. Consequently, every point in $\PP(H^0(\omega_C^{\ot 2})^{\vee})$ corresponds to a symmetric bilinear map on $H^0(\omega_C)$. We claim that the locus where the bilinear map has rank one is precisely $X$. Indeed, we have $\rk(\mu_{\omega_C}^{\vee}(e))=1$ if and only if $\ker(e)=\im(\mu_{\omega_C,V})$, where $V$ is some 1-dimensional subspace of $H^0(\omega_C)$, in which case we have $V=H^0(\omega_C(-P))$ for some $P$. But then $\ker(e)=H^0(\omega_C^{\ot 2}(-P))$; that is, $e$  is the image of $P$ on $X$. It follows that $\mO_C(P)$ is a sub-bundle of $E$, and that $E$ is semi-stable.

\medskip
The same argument also yields the existence of stable rank two bundles with canonical determinant and one global section. 

\medskip
\noindent $\mathbf{g=3,k=3}.$ In this case, $d^*_3=3$. We conclude by applying Corollary~\ref{cor:tri}(1) to a general $\g^1_3$ on a general curve of genus 3. 

\medskip
\noindent $\mathbf{g=4,k=3}.$ In this case, $d^*_3=3$. A general extension of a general $\g^1_4$ on a genus 4 curve by its Serre dual yields a vector bundle with three sections; and since $d^*_3=3$, the bundle is at least semi-stable. If we further assume that the $\g^1_4$, say $(L,H^0(L))$, is base-point free, the bundle is in fact stable. Indeed, if $(r,3)$ is BN-compatible then necessarily $r=1$. On the other hand, if $L$ is base-point free, then its associated complete series has no sub-$\g^1_3$, and hence $E$ has no destabilizing subbundle of degree 3, which means that $E$ is stable.

It is easy to see that a base-point free $\g^1_4$ exists on a general genus 4 curve. Namely, consider the natural (addition) map $W^1_3(C)\times C\to W^1_4(C)\subset \Pic^4(C)$, where $W^1_3(C)$ is the Brill-Noether locus of $\g^1_3$'s inside $\Pic^3(C)$. By the Brill-Noether theorem, the domain is 1-dimensional, while the target $W^1_4(C)$ is 2-dimensional.  

\medskip
\noindent $\mathbf{g=5,k=4}.$ In this case, $d^*_4=4$, and the smallest $d\ge 5$ for which some pair $(r,d)$ is BN-compatible is $d=6$ with $r=2$. Now let $(L,H^0(L))$ be a general $\g^2_6$, and let $e$ be a general extension of $L$ by its Serre dual. By arguing much as in the $g=4$ case, we conclude there is some non-trivial extension $e$ that produces a semi-stable vector bundle with 4 sections. In fact, since $\ker(\mu_L^{\vee})$ is 2-dimensional, the locus $S$ in $\PP(H^0(L^{\ot 2})^{\vee})$ of such extensions $e$ is 1-dimensional. By \cite[Prop. 1.1]{msub}, to conclude that the resulting vector bundle is stable, we need to certify that $e\notin \Sec_2(X)$, where $X$ is the image of $C$ under the morphism $C\to\PP(H^0(L^{\ot 2})^{\vee})$. 

We now claim that $S$ is not contained in $\Sec_2(X)$. To see this, note first that for $e$ to lie on a secant line of $X$ spanned by $x_1$ and $x_2$ means that $\ker(e)\supset H^0(L^{\otimes 2}(-x_1-x_2))$. (In this case $L^{\otimes 2}$ is very ample, so we may identify points on $X$ with points on $C$.) Moreover, $h^0(L^{\otimes 2}(-x_1-x_2))=\dim(\im(\mu_L))=6$. By the Brill-Noether theorem, a $\g^2_6$ on a general genus 5 curve is base-point free, so $\im(\mu_L)\neq H^0(L^{\otimes 2}(-x_1-x_2))$, for any $x_1,x_2$ (not necessarily distinct). It thus suffices to show that there exists a codimension-1 subspace of $H^0(L^{\otimes 2})$ containing $\im(\mu_L)$ but not of the form $H^0(L^{\otimes 2}(-P))$. And indeed, the closed immersion $C\to \PP(H^0(L^{\otimes 2})^{\vee})$ identifies $P\in C$ with $H^0(L^{\otimes 2}(-P))$ (further identified with a line in the dual vector space), while the locus of codimension-1 sub-spaces of $H^0(L^{\otimes 2})$ containing $\im(\mu_L)$ is a line in $\PP(H^0(L^{\otimes 2})^{\vee})$. Since a positive genus curve cannot contain a line, there is some codimension-1 subspace of $H^0(L^{\otimes 2})$ containing $\im(\mu_L)$ but not of the form $H^0(L^{\otimes 2}(-P))$; the claim follows, and we conclude.

\medskip
\noindent $\mathbf{g=6,k=5}.$ In this case, $d^*_5=6$, and we conclude by applying Corollary \ref{cor:tri}(1) to a general $\g^2_6$ on a general curve of genus 6.

\medskip
\noindent $\mathbf{g=7,k=5}.$ In this case, $d^*_5=7$, and we conclude by applying Corollary \ref{cor:tri}(1) to a general $\g^2_7$ on a general curve of genus 7. 

\medskip
\noindent $\mathbf{g=8,k=6}.$ In this case, $d^*_6=8$, and we conclude by applying Corollary \ref{cor:tri}(1) to a general $\g^3_9$ on a general curve of genus 8. This is essentially the example produced by Mukai in \cite{MCGrass}.

\medskip
\noindent $\mathbf{g=9,k=6}.$ In this case, $d^*_6=8$, and the smallest $d\ge 9$ for which some pair $(r,d)$ is BN-compatible is $d=10$ with $r=3$. Similar to the $g=5$ case, we consider the bundle $E$ obtained by a general extension of some $\g^3_{10}$ on a general genus-9 curve by its Serre dual. To conclude that $E$ is stable, it suffices to produce a $\g^3_{10}$ that admits no sub-$\g^2_8$. A well-known theorem of Farkas on inclusion of linear series with base points implies that such $\g^3_{10}$'s exist; see \cite[Thm 0.1]{FSec}. 

\medskip
\noindent $\mathbf{g=10,k=6}.$ Much as in the $g=9$ case, here we apply \cite[Thm 0.1]{FSec} to conclude the existence of a general $\g^3_{11}$ on a genus 10 curve that admits no sub-$\g^2_9$, and then proceed with the same argument as before.

\medskip
\noindent $\mathbf{g=11,k=7}.$ In this case, $d^*_7=13$, and we conclude by applying Corollary \ref{cor:tri}(1) to a general $\g^4_{13}$ on a general curve of genus 11.

\medskip
\noindent $\mathbf{g=12,k=7}.$ In this case, $d^*_7=12$ and we conclude by applying Corollary \ref{cor:tri}(1) to a general $\g^3_{12}$ on a general curve of genus 12 curve.

\subsection{New non-emptiness certificates for BFM loci}
We begin by giving a list of previously unknown cases (of the existence portion) of the BFM conjecture that follow directly from our non-emptiness result for special maximal rank loci. The running numerical assumptions are that (i) $g\ge 13$, $k\ge 8$; and (ii) $\binom{k+2}{2}>3g-3\ge\binom{k+1}{2}$.
It is easy to see that for any $k$ considered here, there is always some positive integer $a$ for which $2a+1+g-d^*_k=k$. We deduce the following result.
\begin{thm}\label{thm:solutions}
The moduli space of stable rank two vector bundles with canonical determinant and $k$ sections over a general curve of genus $g$ is non-empty whenever there exists an integer solution $x=a$ to one of the two following systems of inequalities:
\[
\begin{cases}
\begin{aligned}
&    (x-1)^2-k(x-1)+g<0\\
&   7\ge x^2-7x+2(2k-g+2)>0\\
&    x^2+(7-2k)x+4(g-k-1)\ge 0\\
&2x\ge k
\end{aligned}
\end{cases}
\begin{cases}
\begin{aligned}
&    (x-1)^2-k(x-1)+g<0\\
&    x^2-kx+g\ge 0\\
&    x^2-7x+2(2k-g+2)<0\\
&2x\ge k.
\end{aligned}
\end{cases}
\]
\end{thm}
\begin{proof}
In the system on the left, the first inequality implies that $\rho(g,a-2,2a-3+g-k)<0$; the second inequality implies that $(a-1,2a-1+g-k)$ falls within the surjective range; and the third inequality implies that   
\[\rho(g,a-1,2a-1+g-k)\ge D(g,a-1,2a-1+g-k)\ge 0.\]
In particular, $d^*_k=2a-1+g-k$. Similarly, in the system on the right, we have $d^*_k=2a-1+g-k$ and $(a-1,2a-1+g-k)$ falls within the injective range.

It then follows from Corollary~\ref{cor:tri} and Proposition \ref{prop:positivity_for_small_N} that the corresponding moduli spaces of rank two vector bundles are non-empty.
\end{proof}
Using Theorem \ref{thm:solutions}, we obtain some sharp existence results for the BFM conjecture, some of which were previously unknown.
\begin{cor}\label{cor:BFM_sharp}
The existence portion of the BFM conjecture holds for $g \in \{14,17,18,19,22,26,31\}$. Moreover, there exists a vector bundle verifying the existence portion of the conjecture with a minimal quotient line bundle of degree $d^*_k$ as defined in Definition~\ref{defn:min_deg}. 
\end{cor}
\begin{rem} Those cases in which $g \in\{19, 26, 31\}$ cases were previously settled in \cite{LNP}. In fact, by combining Corollary \ref{cor:BFM_sharp} with results in \cite{Ber}, \cite{Zhang1} and \cite{LNP}, we conclude that whenever $g\le 31$, the existence portion of the BFM conjecture holds except when $g \in \{13,16,20,21,27,28\}$. 
\end{rem}

We now turn our attention to the case $g=13, k=8$, which is of minimal genus among the remaining open cases of the BFM conjecture.
While we do not manage to definitively settle this case, we uncover some interesting geometry in the process. 
\subsection{The case of $g=13$, $k=8$}
For the BFM conjecture to hold in this case, we must produce an extension $e$ of the form 
\[0\to \omega_C\ot L^{-1}\to E\to L\to 0\] 
such that $E$ is stable and $h^0(E)\ge 8$ over a Brill--Noether--Petri general curve. By the theorem of Nagata, it suffices to search within the degree range $[13,18]$. Imposing $h^0(L)+h^1(L)\ge 8$ reduces the possibilities to $d=16$ and $d=18$. 

\medskip
When $d=16$, the Brill--Noether theorem forces $h^0(L)=6$ and $h^1(L)=2$. In this case, any stable $E$ is associated with a nontrivial extension $e$ belonging to the kernel of $H^0(L^2)^{\vee}\to H^0(L)^{\vee}\ot H^0(L)^{\vee}$. In particular, if $E$ is stable, the multiplication map $\mu_L$ cannot be surjective. However, in this case 
\[\rho-1-\bigg|\binom{r+2}{2}-(2d+1-g)\bigg|=1-1-|21-20|<0\]
so the SMRC predicts that $\mu_L$ is {\it always} surjective. Note that the image of any $g^5_{16}$ whose quadratic multiplication map fails to be surjective lies on a Fano threefold $X$ of type $(2,2)$. Explicitly, $\mathbb{P}^3$ is obtained from the blow-up of $X$ along a line $\ell \in X$ via a morphism that contracts the proper transforms of those lines on $X$ that intersect $\ell$. Deciding whether the SMRC holds in this case amounts to a statement about how linear series $g^5_{16}$ transform under the birational isomorphism $X \dashrightarrow \mathbb{P}^3$; we intend to pursue this line further in a subsequent paper. 

\medskip
When $d=18$, the Brill-Noether theorem forces $h^0(L)=7$ and $h^1(L)=1$. By the same argument as in the case $d=16$, if a stable bundle $E$ with sufficiently many sections exists, the corresponding multiplication map $\mu_L$ cannot be surjective. However, we expect the SMRC locus to have dimension
\[\rho-1-\bigg(\binom{r+2}{2}-(2d+1-g)\bigg)=6-1-(28-24)=1\]
so the BFM and SMRC conjectures are compatible in this case. 

Notice that in this case the Brill--Noether theorem also implies that (the complete linear series determined by) any destabilizing subbundle of $E$ must be a $\g^5_{16}$. 
On the other hand, from Proposition~\ref{prop:sec} it follows that no $\g^6_{18}$ giving rise to a stable vector bundle $E$ contains a $\g^5_{16}$. Meanwhile, it follows from Proposition \ref{prop:positivity_for_small_N} that Conjecture~\ref{conj:smrc} (SMRC) holds for $(g,r,d)=(13,6,18)$.
We deduce the following.
\begin{claim}\label{BFM_g=13,k=8}
Assume that Conjecture~\ref{conj:smrc} (SMRC) holds for $(g,r,d)=(13,5,16)$. Then the existence portion of the BFM conjecture holds for $(g,k)=(13,8)$ on a general curve if and only if there exists some $\g^6_{18}$ in $\cM^6_{18}(C)$ which is very ample.
\end{claim}   
By \cite[Thm 0.1]{FSec}, the locus of $\g^6_{18}$ containing a sub-$\g^5_{16}$ is at most 2-dimensional.\footnote{Indeed, the incidence variety $I\subset G^1_8(C)\times \sy^6C$ of degree 6 effective divisors $E$ contained in a $\g^1_8$ is 2-dimensional, and the image of the map $E\mapsto |K_C-E|$ is a 2-dimensional locus inside $G^6_{18}(C)$ consisting of $\g^6_{18}$ containing a $\g^5_{16}$.} So if we can show that the locus in $G^6_{18}(C)$ where the multiplication map $v_2$  fails to have maximal rank is non-empty and is not contained inside the locus of $\g^6_{18}$ admitting a sub-$\g^5_{16}$, we will verify the existence of a rank two linear series with $8$ sections and canonical determinant. 


\bibliographystyle{amsalpha}
\bibliography{myrefs}

\end{document}